\documentclass[11pt]{amsart}
\usepackage{dsfont, amsmath, amsfonts, amsthm, amssymb, times} 
\usepackage[euler-digits]{eulervm}
\numberwithin{equation}{section}

\usepackage{enumerate}
\usepackage{dsfont}
\usepackage{graphicx}

\usepackage{xcolor, soul}
\sethlcolor{green}
\definecolor{ogreen}{rgb}{0,0.6,0}

\usepackage{hyperref}
\hypersetup{
    pdftoolbar=true,
    pdfmenubar=true,
    pdffitwindow=false,
    pdfstartview={FitH},
    pdftitle={Multiple solutions for the van der Waals--Allen--Cahn--Hilliard equation with a volume constraint},
    pdfauthor={V. Benci, S. Nardulli, and P. Piccione},
    pdfsubject={},
    pdfkeywords={},
    pdfnewwindow=true,
    colorlinks=true, 
    linkcolor=blue,
    citecolor=blue,
    urlcolor=blue,
}

\makeatletter
\g@addto@macro\bfseries{\boldmath}
\makeatother

\DeclareMathOperator{\cat}{cat}
\DeclareMathOperator{\argmin}{argmin}
\DeclareMathOperator{\supp}{supp}

\newtheorem{theorem}{Theorem}[]
\newtheorem{lemma}[theorem]{Lemma}
\newtheorem{proposition}[theorem]{Proposition}

\newtheorem{definition}[theorem]{Definition}

\newtheorem*{theorem*}{Theorem}

\theoremstyle{remark}

\newtheorem{remark}[theorem]{Remark}

\allowdisplaybreaks
\numberwithin{equation}{section}
\numberwithin{theorem}{section}

\title[On the van der Waals--Allen--Cahn--Hilliard equation]{Multiple solutions for the\\ van der Waals--Allen--Cahn--Hilliard equation\\ with a volume constraint}
\author[V. Benci]{Vieri Benci}
\author[S. Nardulli]{Stefano Nardulli}
\author[P. Piccione]{Paolo Piccione}


\address{\begin{tabular}{lll}
Universit\`a di Pisa & & Universidade Federal do ABC\\
Dipartimento di Matematica & & Centro de Matem\'atica Cogni\c{c}\~ao Computa\c{c}\~ao\\
Via Filippo Buonarroti 1/c  & & Avenida dos Estados, 5001\\
56123 -- Pisa & &  Santo Andr\'e, SP, CEP 09210-580\\
Italy & & Brazil\\
\emph{E-mail}: {\tt vieri.benci@unipi.it} & &\emph{E-mail}: {\tt stefano.nardulli@ufabc.edu.br}
\\[.5cm]  Universidade de S\~ao Paulo\\ Departamento de Matem\'atica\\ Rua do Mat\~ao 1010\\
S\~ao Paulo, SP 05508--090, Brazil\\ \emph{E-mail}: {\tt piccione@ime.usp.br}
\end{tabular}
}
\date{July 26, 2019}

\thanks{This work was carried out during a visit of V.B.\ at the \emph{Universidade Federal do Rio de Janeiro} and at the \emph{Universidade de S\~ao Paulo}, Brazil. V.B. is very grateful to all faculty and staff at these two institutions, who provided excellent work conditions.}

\begin{document}

\begin{abstract}
We give multiplicity results for the solutions of a nonlinear elliptic equation, with an asymmetric double well potential of Van der Waals-Allen--Cahn--Hilliard type, satisfying a linear volume constraint, on a bounded Lipschitz domain $\Omega\subset\mathds R^N$. The number of solutions is estimated in terms of topological and homological invariants of the underlying domain $\Omega$.
\end{abstract}

\maketitle

\renewcommand{\contentsline}[4]{\csname nuova#1\endcsname{#2}{#3}{#4}}
\newcommand{\nuovasection}[3]{\hbox to \hsize{\vbox{\advance\hsize by -1cm\baselineskip=10pt\parfillskip=0pt\leftskip=3.5cm\noindent\hskip -2.5cm \textbf{#1}\leaders\hbox{.}\hfil\hfil\par}$\,$#2\hfil}}
\newcommand{\nuovasubsection}[3]{\hbox to \hsize{\vbox{\advance\hsize by -1cm\baselineskip=10pt\parfillskip=0pt\leftskip=4cm\noindent\hskip -2cm #1\leaders\hbox{.}\hfil\hfil\par}$\,$#2\hfil}}

\tableofcontents

\section{Introduction}
\subsection{Formulation of the problem ($\mathrm P_{V,\varepsilon}$)}
In this paper we are concerned with the existence of multiple solutions of
the following nonlinear problem ($\mathrm P_{V,\varepsilon}$): for fixed positive constants $V$ and $\varepsilon$, find $
u\in H_{0}^{1}(\Omega )$, and $\lambda \in \mathds{R}$ such that
\begin{equation}\label{Eq:ConstrainedAllenCahn}
\begin{aligned}
-&\varepsilon ^{2}\Delta u+W^{\prime }(u) =\lambda,\\&
\int_{\Omega }u(x)\,\mathrm dx =V,  
\end{aligned}
\end{equation}
where $\Omega$ is an open bounded Lipschitz domain in $\mathds{R}^{N}$, and $W\colon\mathds R\to\mathds R$ is a $C^2$ function which satisfies the following assumptions:
\begin{eqnarray}\label{Eq:Potential}
&& W(0)=W'(0)=0;\quad W''(0)>0,
\\ \label{Eq:Potential2}
&&\exists\ s_0\in\left]0,+\infty\right[\ \text{ s.t.}\quad -m:=W(s_0)=\min\big\{W(s):s\in\mathds{R}\big\}<0,
\\
\label{Eq:WeakUpperBarriers}
&&W'(s)>0,\quad \forall\, s\in\left]s_0, s_0+\delta\right], \text{ for some } \delta>0.
\end{eqnarray} 
In particular \eqref{Eq:WeakUpperBarriers}, is fulfilled for potentials $W$ of class $C^2$ when $s_0$ is a minimum point and 
\begin{equation}\label{Eq:WeakUpperBarriers0}
W''(s_0)>0.
\end{equation}
In what follow we denote by $s_0$ the minimum positive real number for which \eqref{Eq:Potential2} is satisfied.
These assumptions imply that the potential $W$ is not an even function, as opposed to the standard Allen--Cahn potential, which has symmetric minima. 
Moreover, such a $W$ takes different values at the two local minima. We will refer to this situation by saying that $W$ is \emph{asymmetric}.
However, this entails no essential difference in the geometry of the solutions of the problem, see discussion in Section~\ref{sub:discussion}.

For the central result of the paper, we also need an asymptotic growth condition for $W$, given by assuming the existence of positive constants $A$, $B$ such that 
\begin{equation}\label{Eq:Potential0}
\big|W'(s)\big|\le A+B\vert s\vert^{p-1},\quad p<\frac{2N}{N-2}\qquad \text{($p<\infty$ if $N=1,2$)}.
\end{equation}
The graph of a typical potential function $W$ satisfying the above axioms is given in Figure~\ref{fig:graphpotential}.
\begin{figure}
\label{fig:graphpotential}
\begin{center}
\includegraphics[scale=.5]{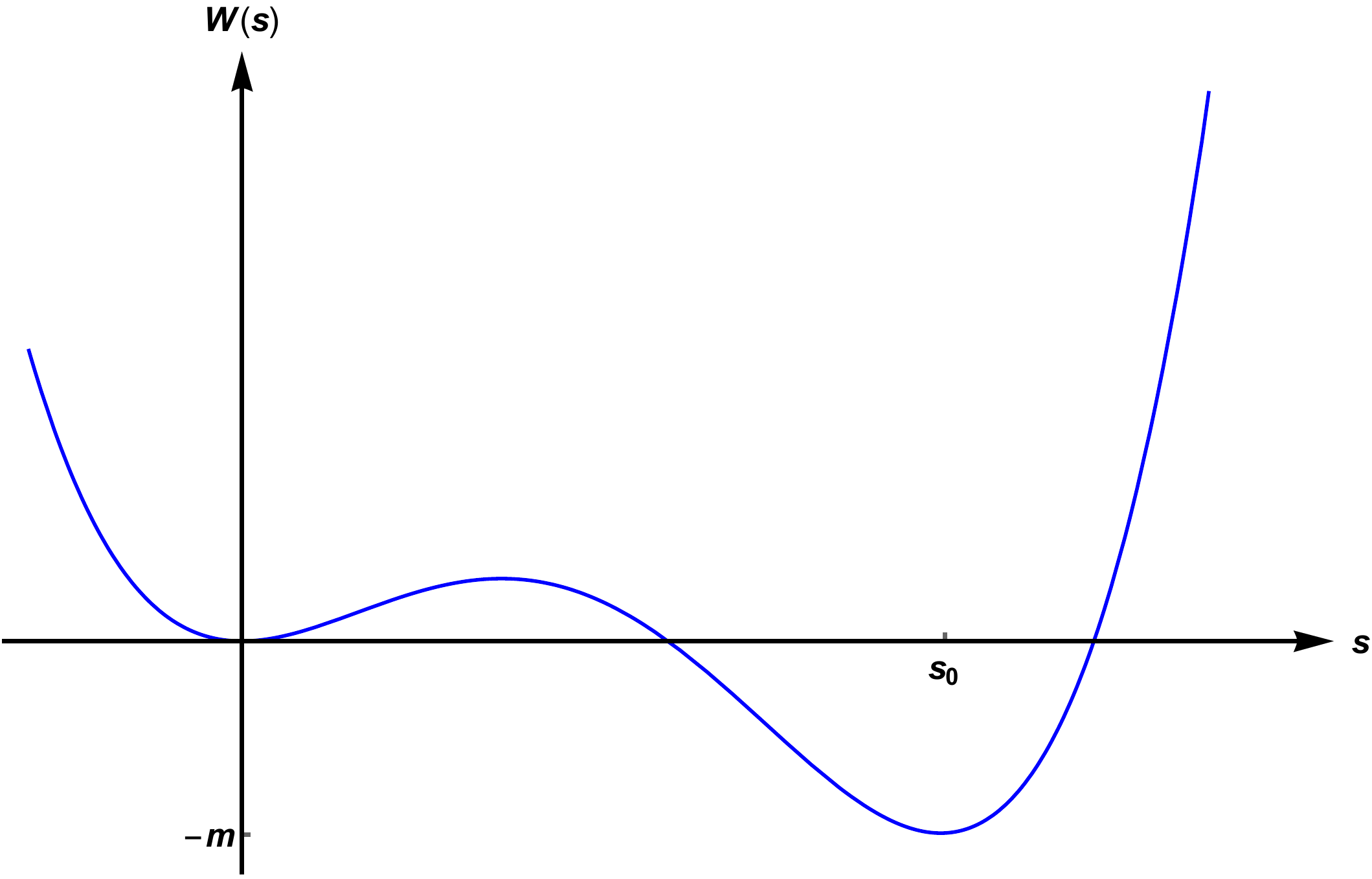}
\caption{The graph of a typical  Van der Waals-Allen--Cahn-Hilliard potential, satisfying \eqref{Eq:Potential}, \eqref{Eq:Potential2}, and \eqref{Eq:WeakUpperBarriers0}.}
\end{center}
\end{figure}%

\begin{remark}\label{thm:rem_no_aypmt_ass}
A simple and instructive example of potentials satisfying \eqref{Eq:Potential}, \eqref{Eq:Potential2}, \eqref{Eq:WeakUpperBarriers} and  is given by the non-symmetric Allen-Cahn-Hilliard potential: 
\begin{equation}\label{Eq:Potential3}
W(s)=s^2(s-s_1)(s-s_2),
\end{equation}%
where $0<s_1<s_0<s_2$. The usual Allen-Cahn potential is taken with $s_1=s_2$. The reader should observe that, when $N\ge3$, assumption~\eqref{Eq:Potential0} is \emph{not} satisfied by \eqref{Eq:Potential3}. However, we will also discuss a result of multiplicity of solutions without assuming the growth condition.
\end{remark}

Equations of type \eqref{Eq:ConstrainedAllenCahn}, such as
the Allen-Cahn equation \cite{AllenCahn1979} and the Cahn-Hilliard equation \cite{CahnHilliard1958} (see also the book  \cite{Presutti2009}), appear naturally in many problems of mathematical physics and applied mathematics. In theoretical biology equations of this type model pattern formation related
to solutions which are not absolute minima of the energy \cite{Murray1981}. From the purely mathematical point of view, equation \eqref{Eq:ConstrainedAllenCahn} is also interesting due to its relation with the theory of constant mean curvature hypersurfaces (cf. \cite{Modica1987}, \cite{HutchinsonTonegawa2000}, \cite{PacardRitore2003}).  Our investigation aims naturally at establishing multiplicity results for mean curvature hypersurfaces via a limiting procedure with the parameter $\varepsilon$ going to zero. The results of the present paper, dealing solely with the case of Euclidean spaces, constitute the first important step; an extension to the general case of (compact) Riemannian manifolds is currently under investigation, see \cite{BenNarOsoPic2018}.
\subsection{The linearized problem}
Assume that $(u,\lambda)\in H^1_0(\Omega)\times\mathds R$ is a solution of \eqref{Eq:ConstrainedAllenCahn}. Linearizing the problem along $u$ gives the following:
\begin{equation}\label{Eq:ConstrainedAllenCahn_lin}
\begin{aligned}
-&\varepsilon ^{2}\Delta \vartheta+W''(u)\vartheta =\Lambda,\\&
\int_{\Omega }\vartheta(x)\,\mathrm dx =0.
\end{aligned}
\end{equation}
\begin{definition}\label{thm:defdegeneratesol}
A solution  $(u,\lambda)$ of Problem \emph{($\mathrm P_{V,\varepsilon}$)} is said to be \emph{degenerate} if \eqref{Eq:ConstrainedAllenCahn_lin} admits a non-trivial solution $(\vartheta,\Lambda)\in H^1_0(\Omega)\times\mathds R$, and \emph{nondegenerate} otherwise.
\end{definition}
It is not hard to see that $(u,\lambda)$ is a nondegenerate solution of ($\mathrm P_{V,\varepsilon}$) when $u$ is a nondegenerate critical point of the associated energy functional, see Section~\ref{sub:variationalframe}.
\subsection{Statement of the existence results}
The focus of this paper is on the existence of solutions which are not
necessarily minima of the associated energy functional (see Section~\ref{sub:variationalframe} below), and on their multiplicity. We recall that in the literature there are many results relative to the existence of multiple solutions which are critical points of the energy. However in these cases, usually it is exploited the fact that $W(0)$ is not a minimum value of $W$ (see the book \cite{RabinowitzCBMS1986}). In other references, multiplicity results are obtained for even potentials $W$, in which case  the topology of the real projective space plays a crucial role. In all these situations, the solutions found present many nodal regions. \smallskip

In the present paper, we find multiple solutions exploiting the
topology of the domain $\Omega$. A lower bound for the number of solutions will be given using Lusternik--Schnirelman theory and Morse theory.\smallskip

 For a topological space $X$, let us denote by $\cat(X)$ the \emph{Lusternik--Schnirelman category} of $X$, see Definition~\ref{Def:Lusternik-SchnirelmannCategory}.

\begin{theorem}\label{Thm:Main1}
Under assumptions \eqref{Eq:Potential}, \eqref{Eq:Potential2}, \eqref{Eq:WeakUpperBarriers}, \eqref{Eq:Potential0}, for $V>0$ sufficiently small, there exists $\varepsilon(V)>0$ such that for all $\varepsilon\in\left]0,\varepsilon(V)\right]$,  Problem \emph{($\mathrm P_{V,\varepsilon}$)} admits:
\begin{itemize}
\item at least one solution if $\Omega$ is contractible;
\item at least $\cat(\Omega )+1$ distinct solutions if $\Omega$ is non-contractible.
\end{itemize} 
Moreover, if $\Omega$ is non-contractible and all solutions of  Problem \emph{($\mathrm P_{V,\varepsilon}$)} are \emph{nondegenerate} $($Definition~\ref{thm:defdegeneratesol}$)$, then there are at least 
$2P_{1}(\Omega)-1$ distinct solutions, where $P_1(\Omega)$ is the sum of the Betti numbers of $\Omega$.  
\end{theorem}
It is a natural conjecture that the nondegeneracy assumption in the last statement of Theorem~\ref{Thm:Main1} should hold for \emph{generic} choices of the quadruple $(\Omega,W,V,\varepsilon)$. It is also interesting that the last claim of Theorem~\ref{Thm:Main1} holds without the nondegeneracy assumption, provided that the solutions are counted with a suitable notion of \emph{multiplicity}, see Definition~\ref{Eq:DefTopologicallyNondegenerate}.\smallskip

The method employed for the construction of the solutions of \eqref{Eq:ConstrainedAllenCahn} also provides bounds for the energy and the Morse index, see Proposition~\ref{thm:boundsenergyindex} below.

\subsection{A brief discussion on the assumptions}\label{sub:discussion}
In the proof of our results, we will use assumptions \eqref{Eq:Potential} to deduce that $W(s)>0$ for $s<0$ (needed in Lemma~\ref{Lemma:WellPosedBarycenters}), that $W(s)\ge -ks$ for some $k>0$ and for $s>0$ small (needed in the proof of Theorem~\ref{Thm:AsymptoticProblemStatement}). Assumption \eqref{Eq:Potential2}, i.e.,  the fact that the absolute minimum of $W$ is negative, is used to deduce that the minimum of the functional \eqref{eq:mainfunctional} is negative, which plays a crucial role in the proof of Theorem~\ref{Thm:AsymptoticProblemStatement}. Namely, this fact will imply that the solution $U_\gamma$ of a certain auxiliary problem (see Section~\ref{sec:auxpb}) has compact support.  Studying the regularity of such a function $U_\gamma$ will require a rather involved analysis of a certain variational inequality, whose solutions are subject to an affine constraint, which is discussed in Sections~\ref{sub:regularityconstr} ,  \ref{sub:existenceradial}, and \ref{sub:asymptotics}. It is important to remark, however, that the fact that $W$ takes different values at the two  local minima, is irrelevant for the geometry of the solutions of the problem. Namely, given a potential $W$ as above, one can consider a linear perturbation of the form $\widetilde W(s)=W(s)+As$, with $A>0$. When $A$ is suitably chosen, the new potential $W$ has two global minima at the zero level; clearly, a pair $(u,\lambda)$ is a solution of Problem ($\mathrm P_{V,\varepsilon}$) with the potential $W$ if and only if $(u,\lambda-A)$ is a solution of Problem ($\mathrm P_{V,\varepsilon}$) with potential $\widetilde W$.

Finally assumption \eqref{Eq:WeakUpperBarriers} is used to guarantee that solutions of the auxiliary minimization problem are bounded from above and \eqref{Eq:Potential} is used to show that solutions of the auxiliary minimization problem are bounded from below. The subcritical growth condition imposed by  \eqref{Eq:Potential0} is needed for technical reason, as it makes the corresponding variational problem well defined in the appropriate Sobolev setting. 

In a forthcoming paper we will develop a theory that allows to obtain a multiplicity result that does not employ the subcritical growth condition~\eqref{Eq:Potential0}. This will be obtained by showing suitable \emph{a priori bounds} for the low energy solutions, including bounds on the corresponding Lagrange multiplier.

\subsection{The variational framework} \label{sub:variationalframe}
Under assumption~\eqref{Eq:Potential0}, solutions of Problem ($\mathrm P_{V,\varepsilon}$) are characterized as critical points of the energy functional \[E_\varepsilon\colon H^1_0(\Omega)\longrightarrow\mathds R\]
defined by:
\begin{equation}\label{eq:mainfunctional}
E_{\varepsilon}(u)=\frac{\varepsilon^{2}}{2}\int_{\Omega} \left\vert
\nabla u\right\vert ^{2}\,\mathrm dx+\int_{\Omega} W\big(u(x)\big)\,\mathrm dx,
\end{equation}
under the constraint 
\begin{equation*}
\int_{\Omega} u(x)\,\mathrm dx=V.
\end{equation*}
Assumption \eqref{Eq:Potential0} guarantees that $E_\varepsilon$ is a well defined functional on $H_0^1(\Omega)$ (see for instance  \cite[Proposition~B.10]{RabinowitzCBMS1986}) which is of class $C^2$. The differential of the functional $E_\varepsilon$ is given by:
\[\phantom{\quad u,v\in H_0^1(\Omega).}E_\varepsilon'(u)v=\varepsilon^2\int_\Omega\nabla u\cdot\nabla v\,\mathrm dx+\int_\Omega W'(u)\cdot v\;\mathrm dx,\quad u,v\in H_0^1(\Omega).\]
Moreover, assumption~\eqref{Eq:Potential2} implies that $E_\varepsilon$ is bounded from below:
\begin{equation}\label{eq:Eeboundedbelow}
\phantom{\quad\forall\,u\in H^1_0(\Omega).}
E_\varepsilon(u)\ge\int_\Omega W\big(u(x)\big)\,\mathrm dx\ge-m\,\vert\Omega\vert,\quad\forall\,u\in H^1_0(\Omega).
\end{equation}
\subsection{Bounds on the energy and the Morse index}
In view to applications to the constant mean curvature problem in
Riemannian manifolds, which requires taking limits to the singular case $\varepsilon\to0$, one needs uniform estimates of the modulus of $\frac{E_\varepsilon}\varepsilon$ and the Morse index of the families of solutions $(u_\varepsilon)_\varepsilon$. The methods developed in the paper allow to obtain the following result:
\begin{proposition}\label{thm:boundsenergyindex}
Under the assumptions of Theorem~\ref{Thm:Main1}, for $V>0$ sufficiently small and for all $\varepsilon\in\left]0,\varepsilon_0(V)\right]$, at least $\cat(\Omega)$ solutions of Problem $(\mathrm P_{V,\varepsilon})$ have energy $\frac{E_\varepsilon}\varepsilon$ which is uniformly bounded in $\varepsilon$. Moreover, in the nondegenerate case, at least $P_1(\Omega)$ solutions of Problem $(\mathrm P_{V,\varepsilon})$ have energy $\frac{E_\varepsilon}\varepsilon$ and Morse index which is uniformly bounded in $\varepsilon$.
\end{proposition}
A proof of Proposition~\ref{thm:boundsenergyindex} will be given at the end of Section~\ref{sec:proofaddhp}.

\section{Notation and preliminary facts}
In this section we present some known results related to the Lusternik--Schnirel\-mann theory and Morse theory which will be used in the sequel.
\begin{definition}\label{Def:Lusternik-SchnirelmannCategory} Let $(X,\tau)$ be a topological space and $Y\subseteq X$ be a closed subset. 
The \emph{Lusternik-Schnirelmann category of $Y$ in $X$} is the number $\cat_{X}(Y)\in\mathds N\bigcup\{+\infty\}$ defined as the minimum number $k$ such that there exist $\mathcal{U}_1,...,\mathcal{U}_k$ open subsets of $X$ contractible in $X$ such that $Y\subseteq\bigcup_i\mathcal{U}_i$. Furthermore, we set $\cat(X):=\cat_X(X)$.
\end{definition}

Let us also recall the following
\begin{definition} Let $\mathfrak{M}$ be a $C^1$-Hilbert manifold, $J\colon\mathfrak{M}\to\mathds{R}$ a $C^1$ functional, and $(u_n)$ a sequence in $\mathfrak{M}$. We say that $u_n$ is a {\emph{Palais--Smale sequence}} $($or a {\emph{PS-sequence}}, for short$)$ for $J$  if 
\begin{equation}\label{Eq:DefPS}
\lim_{n\to\infty}J(u_n)=c\in\mathds R,
\end{equation}
and
\begin{equation}\label{Eq:DefPS0}
\lim_{n\to\infty}\big\Vert J'(u_n)\big\Vert_{T_{u_n}^*\mathfrak{M}}= 0,
\end{equation} 
where $T_{u_n}^*\mathfrak{M}$ denotes the (topological) dual of the tangent space $T_{u_n}\mathfrak{M}$.
\end{definition}
\begin{definition} Let $\mathfrak{M}$ be a $C^2$-Hilbert manifold, $J\colon\mathfrak{M}\to\mathds{R}$ a $C^1$ functional. We say that {\emph{$J$ satisfies the Palais-Smale condition}}, if every Palais-Smale sequence has a convergent subsequence in the \emph{strong} topology of $\mathfrak{M}$. 
\end{definition}
\subsection{Abstract Lusternik--Schnirelman and Morse theory}
To prove our main results we need the following theorem.
\begin{theorem}\label{Thm:FotoLS} Let $\mathfrak{M}$ be a $C^2$-Hilbert manifold and let $J\colon\mathfrak{M}\to\mathds{R}$ be a $C^1$ functional. Assume that 
\begin{enumerate}[$(i)$]
 \item\label{Assumption:Thm:FotoLS} $\inf\limits_{u\in\mathfrak{M}} J(u)>-\infty$;\smallskip
 
 \item\label{Assumption:Thm:FotoLS0}  $J$ satisfies the Palais--Smale condition;\smallskip
 
 \item\label{Assumption:Thm:FotoLS1}there exists a topological space $X$ and two continuous maps $f\colon X\to J^c$, $g\colon J^c\to X$ such that $g\circ f$ is homotopic to the identity map of $X$.
 \end{enumerate} 
 Then there are at least $\cat(X)$ critical points of $J$ in $J^c$. Furthermore, if $\mathfrak{M}$ is contractible and $\cat(X)>1$, or more generally if $\cat(X)>\cat(\mathfrak M)$, there is at least one additional critical point $u\notin J^c$. 
\end{theorem}
\begin{proof} Under assumption \eqref{Assumption:Thm:FotoLS1}, $\cat(X)\le\cat(J^c)$. The result follows by applying standard variational techniques, see \cite{BenciCeramiCalcVar} or \cite{BenciCeramiPassaseo} for details.
\end{proof}
The above result can be improved in the nondegenerate case  using Morse theory.\smallskip

Let $X$ be a topological space and denote by $H^n(X)$ its $n$-th Alexander-Spanier cohomology group with coefficients in $\mathds R$; let $\beta_n(X)$ denote the $n$-th Betti number of $X$, i.e., the dimension of $H^n(X)$. For an account in book form of the Alexander-Spanier cohomology we refer the interested reader to the classical text \cite{MasseyBook}.
\begin{definition}[Poincare's Polynomial]\label{Def:PoincarePolynomial}  The {\emph{Poincare's Polynomial $P_t(X)$ of $X$}} is defined as the formal power series in the variable $t$: 
\begin{equation}\label{Eq:DefPP}
P_t(X):=\sum_{n=0}^{+\infty}\beta_n\,t^n.
\end{equation}
\end{definition}
\begin{remark} If $X$ is a compact manifold, we have that  $H^n(X)$ is a finite dimensional vector space and the formal series \eqref{Eq:DefPP} is actually a polynomial.
\end{remark}
In the following definition we give the notion of Morse index of a critical point, which is necessary in our treatment to establish a relation between the Poincare's polynomial $P_{t}(\Omega )$ and the number of solutions of the Euler--Lagrange equation associated to a given functional $J$. For our purposes, it is necessary to employ an extension of Morse theory to functionals that are not necessarily of class $C^2$, which uses generalized notions of nondegeneracy and Morse index. We will follow here the approach to Morse theory developed in \cite{BenciNewApproach} which is suitable in problems arising from PDE's.  \smallskip

Given a pair $Y\subset X$ of topological spaces and $k\ge0$, let $H^k(X,Y)$ denote the $k$-th relative Alexander--Spanier cohomology group of the pair, and denote by $\beta_k(X,Y)$ its dimension.
\begin{definition}[Morse Index] Let $\mathfrak{M}$ be a $C^2$-Hilbert manifold, $J\colon\mathfrak{M}\to\mathds{R}$ a $C^1$ functional
and let $u\in\mathfrak{M}$ be an isolated critical point of $J$ at level\footnote{This means that $J(u)=c$, $J'[u]=0$, and there exists a neighbourhood $\mathcal{U}$ of $u$ in $\mathfrak{M}$ such that $u$ is the only critical critical point of $J$ in $\mathcal{U}$.} $c\in\mathds R$. We denote by $i_t(u)$ the following formal power series in $t$ 
\begin{equation}
i_t(u):=\sum_{k=0}^{+\infty}\beta_k\big(J^c\cap\mathcal{U}, (J^c\setminus\{u\})\cap\mathcal{U}\big)\,t^k,
\end{equation}
where $J^c=\big\{v\in\mathfrak{M}:\:J(v)\le c\big\}$, and $\mathcal{U}$ is a neighborhood of $u$ containing only $u$ as a critical point. We call $i_t(u)$ the {\emph{polynomial Morse index of $u$}}. The number $i_1(u)$ is called the {\emph{multiplicity of $u$}}.  
\end{definition}
If $J$ is of class $C^2$ in a neighborhood of $u$ and $J''[u]$ is not degenerate, we say that $u$ is a nondegenerate critical point. In this case we have that 
\begin{equation}\label{Eq:DefNondegenerateC^2}
i_t(u)=t^{\mu(u)},
\end{equation} 
where $\mu(u)$ is the Morse index of $u$, i.e., the dimension of a maximal subspace on which the bilinear form $J''[u](\cdot, \cdot)$ is negative-definite. This suggests the following definition.
\begin{definition}\label{Eq:DefTopologicallyNondegenerate} Let $\mathfrak{M}$ be a $C^2$-Hilbert manifold, $J\colon\mathfrak{M}\to\mathds{R}$ be a $C^1$ functional and let $u\in\mathfrak{M}$ be an isolated critical point of $J$ at level $c$. We say that $u$ is \emph{(topologically) nondegenerate}, if $i_t(u)=t^{\mu(u)}$, for some  $\mu(u)\in\mathds{N}$.
\end{definition}

\begin{theorem}\label{Thm:FotoMorse} Let the assumptions \eqref{Assumption:Thm:FotoLS}, \eqref{Assumption:Thm:FotoLS0},  and \eqref{Assumption:Thm:FotoLS1} of Theorem \ref{Thm:FotoLS} hold, and assume additionally that all the critical points of $J^c$ are isolated. Then the following identity of formal power series holds:
\begin{equation}\label{Eq:MorseRelation}
\sum_{u\in Crit(J)}i_t(u)=P_t(X)+t\big[P_t(X)-1\big]+(1+t)Q(t),
\end{equation}
where $Q(t)$ is a polynomial with nonnegative integer coefficients, and $\mathrm{Crit}(J)$ denotes the set of critical points of $J$ on $J^c$. Moreover, if all the critical points are nondegenerate, there are at least $P_1(X)$ critical points with energy less than or equal to $c$, and at least $P_1(X)-1$ critical points with energy greater than $c$.
\end{theorem}
\begin{proof} See \cite{BenciCeramiCalcVar} or \cite{BenciCeramiPassaseo} for details.
\end{proof}
\begin{remark}\label{rem:morseindexlowenergy}
In Theorem~\ref{Thm:FotoMorse}, each one of the $P_1(X)$ critical points with energy less than or equal to $c$ has Morse index that varies in the range $\{0,\ldots,k_*\}$, with $k_*=\max\big\{k\in\mathds N:\beta_k(X)\ne0\big\}$. In particular, when $X$ is a smooth manifold, then the Morse index of these critical points is less than or equal to $\dim(X)$.
\end{remark}
\subsection{Notations}
We will use the following notations throughout the paper:
\begin{itemize}
\item Given $N\ge1$ and a Borel subset $B\subset\mathds R^N$, we will denote by $\big\vert B\big\vert$ the Lebesgue $N$-measure of $B$, and by $\chi_B$ the characteristic function of $B$;\smallskip

\item $\omega_N$ is the volume of the unit ball of $\mathds R^N$, and $\alpha_N$ is the $N$-area of the unit sphere in $\mathds R^{N+1}$;
\smallskip

\item given a functional $\mathcal F$ on a set $\mathcal S$, we will denote by $\argmin\big\{\mathcal F(x):x\in\mathcal S\big\}$ the (possibly empty) set of minimizers of $\mathcal F$ in $\mathcal S$; 
\smallskip

\item given subset $A\subset B\subset\mathds R^N$, we write $A\Subset B$ to mean that the closure of $A$ is compact and it is contained in $B$;
\smallskip

\item for a function $u\colon X\to\mathds R$, we denote by $u^+$ (resp., $u^-$) the \emph{positive $($resp., negative$)$ part} of $u$, defined by $u^+(x)=\max\{u(x),0\}=\frac12\big(u(x)+\vert u(x)\vert\big)$ (resp., $u^-(x)=\max\{-u(x),0\}=\frac12\big(-u(x)+\vert u(x)\vert\big)$);
\smallskip

\item for an integer $k\ge0$ and $\alpha\in\left]0,1\right]$, the H\"older spaces $C^{k,\alpha}(\Omega)$ and $C^{k,\alpha}_\textrm{loc}(\Omega)$, and their Banach space norms, are defined as in \cite[\S~4.1]{GilbargTrudinger}.
\smallskip

\item Given any set $\mathcal X$ and real valued functions $f_1,f_2\colon\mathcal X\to\mathds R$, we define $f_1\vee f_2, f_1\wedge f_2\colon\mathcal X\to\mathds R$ by setting
$f_1\vee f_2(x)=\max\big\{f_1(x),f_2(x)\big\}$ and $f_1\wedge f_2(x)=\min\big\{f_1(x),f_2(x)\big\}$.
\end{itemize}
\section{The auxiliary problem}\label{sec:auxpb}
In order to prove Theorem \ref{Thm:Main1}, we will exploit the properties of a certain solution $U_\gamma$ of an auxiliary variational problem in $\mathds R^N$. Such $U_\gamma$ is a radial function with compact support in $\mathds R^n$, that will be used to define a homotopy inverse for the barycenter map, see Section~\ref{sec:proofaddhp}. 
In order to study the properties of $U_\gamma$, we will employ some results from the classical theory of variational inequalities, for which a standard reference is \cite{KinderlehrerStampacchia}.
\subsection{The auxiliary problem ($\mathrm P_\gamma$)}\label{sub:auxpb}
We consider the following minimization problem ($\mathrm P_\gamma$):  for fixed $\gamma\in\left]0,+\infty\right[,$ find $u\in H^{1}(\mathds{R}^N)$ minimizing 
\[E(u)=\int_{\mathds{R}^N}\left[\tfrac{1}{2}|\nabla u|^2+W(u)\right]\,\mathrm dx,\] over the convex set 
\begin{equation}\label{eq:defKgamma}
K_\gamma=\left\{u\in H^{1}(\mathds{R}^N): u\ge0,\ \int_{\mathds R^N} u\,\mathrm dx\le \gamma\right\},
\end{equation}
where the potential $W\colon\mathds R\to\mathds R$ is defined in the introduction, i.e., it is a map of class $C^2$ satisfying \eqref{Eq:Potential}, \eqref{Eq:Potential2} and \eqref{Eq:Potential0}. Observe that, by Fatou's Lemma, the set $K_\gamma$ is weakly closed\footnote{%
This is the reason for using the constraint $\int_{\mathds R^n}u\,\mathrm dx\le\gamma$, rather than $\int_{\mathds R^N}u\,\mathrm dx=\gamma$. In fact, we will later show that the two constraints define the same minimization problem when $\gamma$ is sufficiently large (see Theorem~\ref{Thm:AsymptoticProblemStatement} , formula \eqref{Eq:AsymptoticProblemStatement0+})}
in $H^1(\mathds R^N)$.  The problem ($\mathrm P_\gamma$) is translation invariant so, if a minimum exists, all its translates are minima too.\smallskip

As it is well known, if a minimum $U_\gamma$ for problem ($\mathrm P_\gamma$) exists, then there exists $\lambda_\gamma=\lambda(U_\gamma) \in \mathds{R}$ such that $U_\gamma$ satisfies the associated variational inequality 
\begin{equation}\label{Eq:AsymptoticEulerLagrangeInequality}
\int\left[\big\langle\nabla U_\gamma,\nabla(v-U_\gamma)\big\rangle+W'(U_\gamma)(v-U_\gamma)\right]\,\mathrm dx\ge\lambda_\gamma\int(v-U_\gamma)\,\mathrm dx,  
\end{equation}
for every $v\in K_\gamma$ (see \cite[Proposition $5.1$, p.\ 15]{KinderlehrerStampacchia}). On the other hand, by \cite[Theorem $2.1$, p.\ $24$, Chapter II]{KinderlehrerStampacchia} applied to the variational inequality \eqref{Eq:AsymptoticEulerLagrangeInequality}, for each fixed $\lambda_\gamma$ and $U_\gamma$ the following variational inequality 
\begin{equation*}\label{Eq:LinearizedVariationalInequality}\int\left[\big\langle\nabla u,\nabla(v-u)\big\rangle+W'(U_\gamma)(v-u)\right]\,\mathrm dx\ge\lambda_\gamma\int(v-U_\gamma)\,\mathrm dx,  \quad\forall\,v\in K_\gamma,
\end{equation*}
admits at most one solution $u$.

\subsection{Analysis of the variational inequality}\label{sub:varineq}
\begin{definition}\emph{(see \cite[Definition 5.1, p.~35, Ch.~II]{KinderlehrerStampacchia})} Let $\Omega\subset\mathds{R}^N$ be an
open subset, $u\in {W^{1,2}(\Omega)}$ and $E\subseteq\overline{\Omega}$. The function $u$ {\emph{is nonnegative on $E$ in the sense of $W^{1,2}(\Omega)$}} if there exists a sequence $u_n\in W^{1,\infty}(\Omega)$ such that 
\begin{equation*}
u_n(x)\ge0,\:\forall x\in E,\quad\text{and}\quad u_n\to u \text{ in } W^{1,2}(\Omega). 
\end{equation*}
We say that {\emph{$u\ge v$ on $E$ in the sense of $W^{1,2}(\Omega)$}}, if $u-v\ge0$ on $E$ in the sense of $W^{1,2}(\Omega)$.
\end{definition}

\begin{definition}\emph{(See 
\cite[Definition 6.7, Ch.~II, p.~45]{KinderlehrerStampacchia})}
\label{Def:PositivityInTheSenseOfStampacchia} Let $\Omega\subset\mathds R^N$ be an open set, $x_0\in\Omega$, and $u\in W^{1,2}(\Omega)$. We say that {\emph{$u(x_0)>0$ in the sense of $W^{1,2}(\Omega)$}}, if there exists an open ball $B_\rho(x_0)$ with $\rho>0$ and $\varphi\in W_0^{1,\infty}\big(B_\rho(x_0)\big)$, $\varphi\ge 0$ and $\varphi(x_0)>0$, such that $u-\varphi\ge0$ on $B_\rho(x_0)$ in the sense of $W^{1,2}(\Omega)$.  For any $\psi\in W^{1,2}(\Omega)$ we say that {\emph{$u(x_0)>\psi(x_0)$ in the sense of $W^{1,2}(\Omega)$}}, if $u(x_0)-\psi(x_0)>0$ in the sense of $W^{1,2}(\Omega)$.   
\end{definition} 
It is easy to see that, for all $u\in W^{1,2}(\Omega)$, the set: \[\Big\{x\in\Omega:u(x)>0\text{ in the sense of } W^{1,2}(\Omega)\Big\}\] is open.\smallskip

Carrying on our analysis of the variational inequalitiy \eqref{Eq:AsymptoticEulerLagrangeInequality}, it is not too hard to prove that\footnote{see \cite[page 43]{KinderlehrerStampacchia} and use a partition of unity argument}, given a minimizer $U_\gamma$ for Problem ($\mathrm P_\gamma$),  on the open subset $\Gamma\subset\mathds{R}^N$:
\begin{equation}\label{eq:defGamma}
\Gamma=\Gamma(U_\gamma)=\big\{x\in\mathds{R}^N:U_\gamma(x)>0\text{ in the sense of } W^{1,2}(\mathds{R}^N)\big\},
\end{equation}
we have that
\begin{equation}\label{Eq:WeakFormulationOnThePositivitySet}
\phantom{\varphi\in C^{\infty}_0(\Gamma)}
\int_\Gamma\left[\big\langle\nabla U_\gamma,\nabla\varphi\big\rangle+W'(U_\gamma)\varphi\right]\,\mathrm dx=\lambda_\gamma\int_\Gamma\varphi\,\mathrm dx,\quad \forall\varphi\in C^{\infty}_0(\Gamma).
\end{equation}
We also need the following notation 
\[H^1_\text{rad}(\mathds{R}^N)=\big\{u\in H^1(\mathds{R}^N): u \text{ is radially symmetric}\big\}.
\]
\begin{definition}\label{Def:SymmetricDecreasingRearrangement} Let $A$ be a measurable set of finite volume in $\mathds{R}^n$. Its {\emph{symmetric rearrangement}} $A^*$ is the open ball centered at the origin whose volume agrees with the volume of $A$. Let $f\colon\mathds R^N\to\mathds R$ be a nonnegative measurable function that {\emph{vanishes at infinity}}, in the sense that all its superlevel sets have finite measure, i.e., $\big|\{x:f(x)>t\}\big|<+\infty$, for all  $t>0$. The {\emph{symmetric decreasing rearrangement of $f$}} is the radially symmetric function $f^*$ whose superlevel sets are the symmetric rearrangements of the superlevel sets of $f$. Thus: 
$f^*(x)=\int_0^{+\infty}\chi_{\{y:f(x)>y\}^*}(t)\,\mathrm dt$. 
\end{definition}
The symmetric decreasing rearrangement $f^*$ of a measurable function $f$ is lower semicontinuous (since its level sets are open), and it is uniquely determined by the distribution function $\mu_f(t):= |\{x : f(x) > t\}|$. By construction, $f^*$ is \emph{equimeasurable} with $f,$ i.e., corresponding superlevel sets of $f$ and of $f^*$ have the same volume, $\mu_f(t) =\mu_{f^*}(t)$ for all $t>0$.
\begin{lemma}\label{Lemma:EnergyDecreasingUnderScwartzSymmetrization} $E(u^*)\le E(u)$, for every $u\in W^{1,2}(\mathds{R}^N)$, $u\ge0$.
\end{lemma}
\begin{proof} By the Polya-Szego inequality we have that $\int|\nabla u^*|^2\,\mathrm dx\le\int|\nabla u|^2\,\mathrm dx$. Using the Layer-Cake integral representation of a nonnegative function we have that $\int W(u^*)\,\mathrm dx=\int W(u)\,\mathrm dx$ (compare with \cite[Proposition ~2.6]{VanSchaftingen}). The conclusion follows easily.
\end{proof}
Let us recall the following result from \cite{BrothersZiemer}.
\begin{theorem}\label{Thm:BrothersZiemer} Let $u\colon\mathds{R}^N\to\left[0,+\infty\right[$ be a function with $\supp(u)\Subset\mathds{R}^N$, and let $A\colon\left[0,+\infty\right[\to\left[0,+\infty\right[$ be a function  of class $C^2$, such that $A(0)=0$, and with $A^{\frac1p}$ convex for some $1\le p<+\infty$. Assume $\int_{\mathds{R}^N}A(|\nabla u|)\,\mathrm dx<+\infty$. Then $\nabla u^*$ is a measurable function and 
\begin{equation}\label{Eq:BrothersZiemerStatement}
\int_{\mathds{R}^N}A\big(|\nabla u^*|\big)\,\mathrm dx\le\int_{\mathds{R}^N}A\big(|\nabla u|\big)\,\mathrm dx.
\end{equation}
Moreover, if $p>1$, if
\begin{equation}\label{Eq:BrothersZiemerStatement0}
\Big|\nabla {u^*}^{-1}(0)\cap {u^*}^{-1}\big(\left]0,\Vert u^*\Vert_\infty\right[\big)\Big|=0,
\end{equation} 
if $A$ is strictly increasing, and if equality holds in \eqref{Eq:BrothersZiemerStatement}, then there exists $x_0\in\mathds{R}^N$ such that $u^*(x_0+x)=u(x)$ a.e. in $\mathds{R}^N$. 
\end{theorem}
\begin{proof}
See \cite[Theorem~1.1]{BrothersZiemer}.
\end{proof}
\begin{theorem}\label{Thm:Petrosyan} Let $D$ be an open bounded domain of $\mathds{R}^N$, $u\in L^1(D)$, $f\in L^p(D)$, $1 < p <\infty$, be such that $\Delta u = f$, in $D$ in the sense of distributions. Then $u\in W_{\textrm{loc}}^{2,p}(D)$ and there exists a constant $C = C(p,n,K,D)>0$ such that:
\begin{equation}
||u||_{W^{2,p}(K)}\le C\left\{||u||_{L^1(D)}+||f||_{L^p(D)}\right\},
\end{equation}
for any $K\Subset D$. In particular, $\Delta u=f$, a.e.\ $\Omega$.
\end{theorem}
\begin{proof}
See \cite[Theorem~1.1]{PetrosyanShahgholianUraltseva}.
\end{proof}
It will also be useful to keep in mind standard elliptic regularity results, such as \cite[Theorem~9.19]{GilbargTrudinger}, that will play an important role in the proof of Theorem~\ref{Thm:AsymptoticProblemStatement} below.
Let us also recall a celebrated result of Gidas, Ni and Nirenberg concerning the symmetry of solutions of certain elliptic PDEs:
\begin{theorem}\label{Thm:GidasNiNirenberg}
Let $f$ be of class $C^1$ and let $u>0$ be a positive solution in $C^2\big(\overline{B_{\mathds{R}^N}(0,R)}\big)$ of $\Delta u+f(u) =0$ with $u=0$ on $|x|=R$. Then $u$ is radially symmetric, and $\frac{\partial u}{\partial r}<0$ for $0<r<R$. In particular $||u||_\infty=u(0)$.
\end{theorem}
\begin{proof}
See \cite[Theorem~1]{GidasNiNirenberg}.
\end{proof}
\subsection{Regularity of the obstacle problem under the volume constraint}\label{sub:regularityconstr}
We will need a regularity result for solution of variational problems with constraints.
The following theorem is obtained with a slight modification of the arguments used in the proof of \cite[Thm.~1, Thm.~2]{Eisen83}. Compared with the results of \cite{Eisen83}, here we consider the case where an extra non-homogeneous term is present. For our purposes this extra term denoted by $f$ is in $L^\infty$, and it only depends on the unknown function $u$ (and not on its gradient $\nabla u$).
For the sake of completeness, we will prove here a statement which is more general than the one we need in the proof of Theorem~\ref{Thm:AsymptoticEquivalentOfTheRadius}. 
\smallskip

Let us consider a bounded Lipschitz domain $\Omega\subset\mathds R^N$; given a constant $V\in\mathds R$ and functions $\psi_1,\psi_2,\chi\in H^1( \Omega )$ with 
$\psi_1\le \chi \le \psi_2$ and:
\[\int_\Omega\psi_1\, dx<V<\int_\Omega\psi_2\, dx,\]
let us denote by
\[
\mathbf K_{\psi_1,\psi_2, \chi,V,\Omega}= \left\{ v\in H^1( \Omega ) : (v-\chi) \in H^1_0(\Omega),\ \int _ \Omega  v\; dx= V,\  \psi_1\leq v\leq\psi_2\right\}.
\]
In our next result, we will consider only the case where the function $\psi_1$ is constant.
\begin{theorem}\label{Thm:FreeBoundaryRegularityForDerivative}
 If $u\in\mathbf K_{\psi_1,\psi_2, \chi,V,\Omega}$ is a solution of the variational inequality:
\begin{equation}\label{Eq:Eisen1}
\int_{\Omega} \nabla u\cdot\nabla(u-v) \, d x \leq \int _ { \Omega } f(x)\big(u(x) - v(x)\big) dx,
\end{equation}
for all $v\in\mathbf K_{\psi_1,\psi_2, \chi,V,\Omega},$ where $\psi_1$ is a constant function, and
$$
f\in L^\infty( \Omega ),\quad\psi_2\in C^{1,\alpha}_{\text{loc}}(\Omega),$$ then $\nabla u\in C^{0,\alpha}_{\text{loc}}(\Omega)$, i.e., $u\in C^{1,\alpha}_{\text{loc}}(\Omega)$. Moreover, if $\psi_2\in C^{1,\alpha}(\overline{\Omega})$, then $u\in C^{1,\alpha}(\overline{\Omega})$.
\end{theorem} 
\begin{remark} Theorem \ref{Thm:FreeBoundaryRegularityForDerivative} is an essential regularity results, that will needed in the proof of Theorem \ref{Thm:AsymptoticEquivalentOfTheRadius} to establish that the radial solution of a certain auxiliary problem has vanishing normal derivative along the boundary of its support.  
\end{remark}
\begin{proof}[Proof of Theorem~\ref{Thm:FreeBoundaryRegularityForDerivative}]
For the sake of brevity, we will denote $\mathbf{K}:=\mathbf K_{\psi_1,\psi_2, u,V,\Omega}$.
For every subset $A\subseteq\mathds R^n$, $|A|$ denotes the Lebesgue's measure of $A$. When 
$$
\Big|\big\{ x \in \Omega : \psi _ { 1 } < u(x) < \psi _ { 2 }(x)\big \}\Big|=0,
$$
the statement of the theorem becomes trivial since it means that $u=\psi_1$, a.e., or $u=\psi_2$, a.e. Thus from now on we can assume that
$$
\Big|\big\{ x \in \Omega : \psi _ { 1 } < u(x) < \psi _ { 2 } ( x )\big \}\Big|> 0;
$$
thus, we get the existence of $\varepsilon_0=\varepsilon_0(u, \psi_1, \psi_2)>0$ such that
$$
\Big|\big\{ x \in \Omega : \psi _ { 1 } + \varepsilon_0<u(x)<\psi_{2}(x) - \varepsilon_0 \big\}\Big|>0.
$$
Now we construct a function $\hat{\varphi}\colon\mathds R\to[0,1]$, $\hat{\varphi}\in C^\infty(\mathds R)$, such that $\hat{\varphi}\vert_{\left]-\infty,\varepsilon_0\right]}=0$ and $\hat{\varphi}\vert_{\left[2\varepsilon_0,+\infty\right[}=1$. We can find a small ball $B\Subset\Omega$ and a function $\nu\in W_0^{1,2}(B)$, satisfying $0 \leq \nu \leq 1 ,$ with the property that
\begin{equation}\label{Eq:EisenPropertiesofvarphi}
\varphi(x):=\nu(x)\cdot\hat{\varphi}\left(u(x)-\psi_{1}(x)\right)\cdot\hat{\varphi} \left(\psi_{2}(x)-u(x)\right)
\end{equation}
does not vanish identically, and
\begin{equation}\label{Eq:EisenPropertiesofvarphi0}
\int_B\varphi(x)\, dx>0.
\end{equation}
From the construction of $\varphi$ we have
$$
\psi _ { 1 } \leqq u + t \varphi \leqq \psi _ { 2 } \ \text { a.e.\ in }\Omega,\quad\forall\, t\in\left]-\varepsilon_0,\varepsilon_0\right[.
$$
Now let $B _ { R } \left(x_0 \right) \subset \Omega ,$ where $R > 0$ is arbitrary, and $B _ { R } \left(x_0 \right) \cap B = \emptyset$.

\begin{remark} 
In order to obtain the following parts of the proof for all $B _ { R } \left(x_0 \right) \subset\Omega$ we need the existence of two disjoint small balls $B^1$, $B ^ { 2 }$, and two functions $\varphi _ { 1 } \in H^ { 1 } \left( B ^ { 1 } \right) , \varphi _ { 2 } \in H^ { 1 } \left( B ^ { 2 } \right)$ with the same properties as $\varphi$, i.e., satisfying \eqref{Eq:EisenPropertiesofvarphi} and \eqref{Eq:EisenPropertiesofvarphi0}. 
The existence of such function can be shown by simply repeating the existence argument above, with $B^1$ small enough. Then, for $R_0>0$ sufficiently small,  we have $B _ { R _ { 0 } } \left(x_0\right) \subset \Omega$, and either $B_{R_0}\left(x_0 \right) \cap B^1=\emptyset$ or $B_{R_0} \left( x_0 \right) \cap B ^ { 2 } = \emptyset$.  
\end{remark}
For sake of simplicity we deal only with $B$ and $\varphi$, for a more detailed treatment of this standard isoperimetric argument compare \cite[Example~2.13, p.\ 279--280]{Maggi2012}. Choose $R_0$ small enough as prescribed by the preceding remark and set $B:=B^i$ and $\varphi:=\varphi_i$, where $i\in\{1,2\}$ is such that $B_{R_0}\left(x_0 \right) \cap B^i=\emptyset$. We want to show the desired regularity of $U$ inside the ball $B_R(x_0)$ for any $0<R<R_0$ and $x_0\in\Omega$. With this aim in mind let $U$ be the harmonic function on $B _ { R }:=B_R(x_0)$ with the boundary values of $u$, i.e., 
\begin{equation}\label{Eq:Eisen5}
\int _ { B _ { R } } \nabla U \nabla\phi d x = 0,\forall\phi\in W^{1,2}_0(B_R), \text{ and } U-u\in W^{1,2}_0(B_R).
\end{equation}
By a standard argument (see  \cite[Lemma 7.I]{Campanato65}) we have for $0 < \rho < R $ 
\begin{equation}\label{Eq:Eisen6}
\int _ { B_\rho } | \nabla U | ^ { 2 } d x\le\left( \frac { \rho } { R } \right) ^ { n }\underset { B _ { R } } { \int } | \nabla U| ^ { 2 } d x.
\end{equation}
Furthermore we introduce
\[\widetilde U=\begin{cases}U,&\text{in $B_R$};\\ u,&\text{in $\Omega\setminus B_R$}.\end{cases}\]
and we set $v:=\left(\tilde{U} \vee \psi _ { 1 } \right) \wedge \psi _ { 2 }$. By construction we know that there is $i\in\{1,2\}$, such that $B_i\cap B_R=\emptyset$. Set $\varphi:=\varphi_i$, $B:=B_i$, and let $\tilde{t}:=\tilde{t}_{B_R}\in\mathds R$ verifying 
$$
\int _ { \Omega } (v + \tilde { t } \varphi ) d x = V=\int _ { \Omega\setminus B _ { R } } u+\int _ { B _ { R }} u=\int _ { \Omega\setminus B _ { R } } v+\int _ { B _ { R }} v+\tilde { t }\int_B\varphi.
$$
It follows that 
$$
\tilde{t}=\int _ { B _ { R } } \left( u - \left( U\vee\psi _ { 1 } \right) \wedge \psi _ { 2 } \right) dx\left( \int_B \varphi dx \right) ^ { - 1 }.
$$
As it is immediate to check, inside the ball $B_R$ and also in the entire $\mathds R^n$ we have 
\begin{equation}\label{Eq:MajoratingTrick}
|u-v|\le|u-\tilde{U}|.
\end{equation} 
This simple observation allow us to estimate the value of $\tilde{t}$, i.e.,
\begin{equation}\label{Eq:Eisen9}
| \tilde { t } | \leq \int _ { B _ { R } } | u - U | d x \left( \int _ { B } \varphi d x \right) ^ { - 1 }.
\end{equation}
We claim that if we choose $0<R_0$ small enough then for every $0<R<R_0$ we have $|\tilde{t}_{B_R}|<\varepsilon_0$. By an application of H\"older inequality and Gagliardo-Nirenberg inequality (which is possible because $U-u\in W^{1,2}_0(B_R)$) we have that 
\begin{equation}\label{Eq:Eisen10}
\underset { B _ { R } } { \int } | u - U | d x \leq(\omega_nR^n)^{\frac1n+\frac12}\left( \int _ { B_R } | \nabla ( u - U ) | ^ { 2 } d x \right) ^ { \frac { 1 } { 2 } }.
\end{equation}
On the other hand
$$
\underset { B _ { R } } { \int } | \nabla U | ^ { 2 } d x\le\underset { B _ { R } } { \int } | \nabla u | ^ { 2 } d x,
$$ 
since $\int_{B_R}\nabla U\cdot\nabla u=\underset { B _ { R } } { \int } | \nabla U | ^ { 2 } d x$ and $$-\int_{B_R}\nabla U\cdot\nabla u+\int_{B_R}|\nabla u|^2=\int_{B_R}|\nabla (U-u)|^2\ge0.$$
Hence
\begin{equation}\label{Eq:Eisen11}
\underset { B _ { R } } { \int } | u - U | d x \leq(\omega_nR^n)^{\frac1n+\frac12}\left( 2 \int _ { \Omega } | \nabla u | ^ { 2 } d x \right)^{\frac12}.
\end{equation}
From \eqref{Eq:Eisen9} and \eqref{Eq:Eisen11} we get easily
$$
| \tilde { t } | \leq(\omega_nr^n)^{\frac1n+\frac12}\left( 2 \int _ { \Omega } | \nabla u | ^ { 2 } d x \right)^{\frac12}\left( \int _ { B } \varphi d x \right) ^ { - 1 }.
$$
From the last inequality we see that for a suitable $R_0=R_0(\varphi, u)>0$ we have that for any $0<R<R_0$ it holds $|\tilde{t}(B_R,\varphi, \psi_1,\psi_2, u)|<\varepsilon_0$. This readily implies that for every $0<R<R_0$ we have $v+\tilde{t}\varphi\in\mathbf{K}$. Thus
$$
\int _ { \Omega } \nabla u \nabla ( u - v - \tilde{t} \varphi ) d x \le\int_\Omega f(x)(u-v)(x)dx+\tilde{t}\int_{B}f\varphi dx.
$$
At first we reduce the problem to the case 
$$
\psi _ { 1 } = 0.
$$
In fact, $u \in K$ is a solution of \eqref{Eq:Eisen1} with $f_i=0,\forall i\in\{1,...,n\}$ if and only if $\overline { u } = \left( u - \psi _ { 1 } \right) \in \overline{K}$ solves
$$\int_{\Omega} \nabla \overline {u} \nabla ( \overline {u} - \overline {v} ) d x \leq\int _ { \Omega } f(u-v)dx,$$ for all 
\begin{multline*}
\overline {v}\in\overline{ K } : = \Big\{ \tilde{v}\in W^{1,2}( \Omega ) | \tilde{v}- \overline{u}\in W_0^{1,2}( \Omega ),\quad \int _ { \Omega }\tilde{v}d x = V - \int _ { \Omega } \psi _ { 1 }\,  d x,\\ 0 \leq\tilde{v}\le\psi=\psi_2-\psi_1\}\Big\}.
\end{multline*}

Let $\overline { U }$ be the harmonic function on $B_R:=B_R(x_0)\Subset\Omega$ with boundary values $u,$ then for $0 < \rho<R\le R_0$
\begin{equation}\label{Eq:Eisen20}
\int_{B_\rho}\left| \nabla \overline { U } - ( \nabla \overline { U } ) _ { \rho } \right| ^ { 2 } d x\le c_ { 4 } \left( \frac { \rho } { R } \right) ^ { n  + 2 } \int _ { B_R } \left| \nabla \overline { U } - ( \nabla\overline{U})_{R}\right| ^ { 2 } d x,
\end{equation}
where $( \nabla \overline { U } )_\rho:=\frac1{|B_\rho(x_0)|}\int_{B _ { \rho }(x_0)} \nabla \overline { U } d x$, and $c_4=c_4(n)>0$. We set
$$
\overline{v} = \left\{ \begin{array} { l l l } {\overline{u}} & { \text { on } } & { \Omega > B _ { R } },\\ { \overline {U} \wedge \psi } & { \text { on } } & B_R.\end{array} \right.
$$
With $\varphi$ and $\tilde { t }$ as before we have that $\overline{v} + \tilde { t } \varphi \in \overline { K }$ and for
$0 < R \leq R_0$,
\begin{eqnarray*}
\int_{ B _ { R } } \nabla \overline { u }\cdot \nabla ( \overline { u } - \overline { U } \wedge \psi ) & \le & \tilde{t} \int _ { B } \nabla \overline{u}\cdot \nabla \varphi d x\\ 
& + & \int_{B_R} f(x)(\overline {u}-\overline {U} \wedge \psi)(x)+\tilde{t}\int_{B_R}f\varphi dx.
\end{eqnarray*}
Recall the following easy inequality:
\begin{multline}\label{Eq:NonHomogeneusTermEstimates}
\left|\int_{B_R} f(x)(u-v)(x)dx\right|  \stackrel{\eqref{Eq:MajoratingTrick}}{\le}  ||f||_{\infty,\Omega}\int_{B_R}|u-U|dx\\ 
\stackrel{\eqref{Eq:Eisen11}}{\le}  ||f||_{\infty,\Omega}(\omega_nR^n)^{\frac1n+\frac12}\left( \int _ { B_R } | \nabla ( u - U ) | ^ { 2 } d x \right) ^ { \frac { 1 } { 2 } }.
\end{multline}
We write $\overline { u } - \overline { U } \wedge \psi = \overline { u } - \overline { U } + \overline { U } - \overline { U } \wedge \psi ,$ we make use of
$$
0=\int_{B_R} - \gamma \nabla ( \overline { u } - \overline { U } \wedge \psi ) d x
$$
which is true for all $\gamma \in \mathds { R } ^ { n } ,$ and observe that \eqref{Eq:Eisen5}, \eqref{Eq:Eisen9}, \eqref{Eq:Eisen10}, and \eqref{Eq:NonHomogeneusTermEstimates} still hold when applied to $\overline{u}, \overline{v}, \overline{U}$. 
Then we get
\begin{eqnarray*}
\int_{ B _ { R } } \nabla \overline { u } \nabla ( \overline { u } - \overline { U } \wedge \psi ) & \le & \tilde{t} \int _ { B } \nabla \overline{u} \nabla \varphi d x\\
& + & \int_{B_R} f(x)(\overline {u}-\overline {U} \wedge \psi)(x)+\tilde{t}\int_{B_R}f\varphi dx.\\
\end{eqnarray*}
In fact we are able to prove by elementary meanings the following inequality just in the case $\psi_1=const.\in\mathds R$ (which implies $\overline { U }=U$)
\begin{eqnarray*}
\int _ { B _ { R } } | \nabla (\overline{u}-\overline{U})|^2dx & \le & \tilde{t} \int _ { B } \nabla \overline{u} \nabla \varphi d x\\
& + & \int_{B_R} f(x)(\overline {u}-\overline {U} \wedge \psi)(x)+\tilde{t}\int_{B_R}f\varphi dx\\
& + & \int_{ B _ { R } } \nabla \overline { u } \nabla ( \overline {U} \wedge \psi- \overline{U}).
\end{eqnarray*}
Hence
\begin{equation}\label{Eq:Eisen22}
\int _ { B _ { R } } | \nabla ( \overline { u } - \overline { U } ) | ^ { 2 } d x\le c _ { 5 } \left\{ \int _ { B _ { R } } | \nabla ( \overline { U } - \overline { U } \wedge \psi )|^2dx+\int _ { B _ { R } } \left| \nabla \psi _ { 1 } - \gamma \right| ^ { 2 } d x + R ^ { n + 2 }\right\},
\end{equation}
where $c_ { 5 }=c_5(n,\varphi, f, u)>0$ depends on $\tilde{C}$ and $n$.
To estimate
$$
\int _ { B_R }  | \nabla ( \overline { U } - \overline { U } \wedge \psi ) | ^ { 2 } d x,
$$
we observe that 
$$
\int_{B_R} \nabla ( \overline { U } - \overline { U } \wedge \psi)\nabla\phi dx=-\int_{B_R}(\nabla(\overline{U}\wedge\psi)-\delta)\cdot\nabla\phi dx
$$
being true for each $\delta\in\mathds R^n$ and each $\phi\in W^{1,2}_0(B_R)$. We choose $\phi:= ( \overline { U } - \overline { U } \wedge\psi)$ and get
\begin{eqnarray}\label{Eq:Eisen23}\nonumber
\underset { B _ { R } }\int | \nabla ( \overline { U } - \overline { U } \wedge \psi ) | ^ { 2 } d x & = & \int_{\{\overline{U}>\psi\}} | \nabla ( \overline { U } - \overline { U } \wedge \psi ) | ^ { 2 } d x\\ 
& - & \int_{\{\overline{U}>\psi\}}(\nabla(\overline{U}\wedge\psi)-\delta)\cdot\nabla ( \overline { U } - \overline { U } \wedge\psi)dx\\ \nonumber
& = & \int_{\{\overline{U}>\psi\}}(\nabla(\psi)-\delta)\cdot\nabla ( \overline { U } - \psi)dx\\ \nonumber
& \stackrel{H\text{\"o}lder}{\le} & \{\int_{\{\overline{U}>\psi\}}|\nabla(\psi)-\delta|^2\}^{\frac12}\{\int_{\{\overline{U}>\psi\}}|\nabla (\overline { U } - \psi)|^2\}^{\frac12}.
\end{eqnarray}
Dividing by $\{\int_{\{\overline{U}>\psi\}}|\nabla (\overline { U } - \psi)|^2\}^{\frac12}$ yields to 
\begin{eqnarray*}
\underset { B _ { R } }\int | \nabla ( \overline { U } - \overline { U } \wedge \psi ) | ^ { 2 } d x & = & \int_{\{\overline{U}>\psi\}}|\nabla (\overline { U } - \psi)|^2\\
 & \le & \int_{B_R} | \nabla \psi - \delta |^2dx.
\end{eqnarray*}
Combining \eqref{Eq:Eisen22} and \eqref{Eq:Eisen23} we deduce
\begin{equation}\label{Eq:Eisen24}
\int_{B_R} | \nabla ( \overline { u } - \overline { U } ) | ^ { 2 } dx \le c_5\left\{\int_{B_R} | \nabla \psi - \delta | ^ { 2 } d x+\int_{B_R} |\gamma| ^ { 2 } d x+R^{n+2}\right\}.
\end{equation} 
As $\psi_1,\psi_2,\psi\in C^{1,\alpha}_{\text{loc}}(\Omega),$ it is obvious that choosing $\delta = ( \nabla \psi ) \left(x_0\right)$ and $\gamma=0$ we obtain
\begin{equation}\label{Eq:Eisen25}
\int_{B_R} | \nabla ( \overline { u } - \overline { U } ) | ^ { 2 } dx \le c_6R ^ {n + 2 \alpha },
\end{equation}
for $0<R\leq R _ { 0 } \le 1$. From this, we conclude that 
\begin{multline*}
\left|\int_{B_R}|\nabla\overline{u}-(\nabla\overline{u})_R|^2dx-\int_{B_R}|\nabla\overline{U}-(\nabla\overline{U})_R|^2dx\right| \\\le \tilde{c}_6R^{n+2\alpha}+\int_{B_R} | (\nabla\overline{u})_\rho - (\nabla\overline { U })_\rho ) | ^ { 2 }dx \\
 \stackrel{\text{Jensen}}{\le}  \tilde{c}_6R^{n+2\alpha}+\int_{B_R} | \nabla ( \overline { u } - \overline { U } ) | ^ { 2 }dx,
\end{multline*}
this last equation together with \eqref{Eq:Eisen20} for $0 < \rho \leq R$
\begin{eqnarray}\label{Eq:Eisen26}\nonumber
\int_{B_\rho}|\nabla\overline{u}-(\nabla\overline{u})_\rho|^2dx\hspace{-2cm} & \le  2\left(\frac\rho R\right)^{n+2}\int_{B_R}|\nabla\overline{U}-(\nabla\overline{U})_R|^2dx\\\nonumber
&+ 2\int_{B_\rho} | \nabla ( \overline { u } - \overline { U } ) | ^ { 2 } dx\\ \nonumber
& +  2\int_{B_\rho} | (\nabla\overline{u})_\rho - (\nabla\overline { U })_\rho ) | ^ { 2 } dx\\ \nonumber
& \stackrel{\eqref{Eq:Eisen25}}{\le}  2\left(\frac\rho R\right)^{n+2}\int_{B_R}|\nabla\overline{U}-(\nabla\overline{U})_R|^2dx+2c_6R^{n+2\alpha}\\ 
& {+}  2\int_{B_R} | \nabla ( \overline { u } - \overline { U } ) | ^ { 2 } dx\\ \nonumber
& \stackrel{\eqref{Eq:Eisen25}}{\le}  2\left(\frac\rho R\right)^{n+2}\int_{B_R}|\nabla\overline{U}-(\nabla\overline{U})_R|^2dx+4c_6R^{n+2\alpha}\\ \nonumber
& \stackrel{\text{Jensen and }\eqref{Eq:Eisen25}}{\le}   c_7\left\{\left(\frac\rho R\right)^{n+2}\int_{B_R}|\nabla\overline{u}-(\nabla\overline{u})_R|^2dx+ R^{n+2\alpha}\right\}.
\end{eqnarray}
Again Lemma $6.1$ of \cite{Campanato65} or Lemma $2.1$ of \cite{Giaquinta2016multiple}, or Lemma 8.23 di \cite{GilbargTrudinger} combined with \eqref{Eq:Eisen26} imply
\begin{equation}\label{Eq:CampanatoEstimateForDerivatives}
\int_{B_\rho}|\nabla\overline{u}-(\nabla\overline{u})_\rho|^2dx\le c_8\left(\frac\rho R\right)^{n+2\alpha}(\int_\Omega|\nabla\overline{u}|^2+1),
\end{equation}
and a well-known result of Campanato (see also \cite{Giaquinta2016multiple}) says that $\overline {{u}}$, and hence $u$, belongs to $C^{1,\alpha}_{\text{loc}} (\Omega)$ and also that $u\in C^{1,\alpha}(\overline{\Omega})$, provided $\psi_2\in C^{1,\alpha}(\overline{\Omega})$ and $f\in L^\infty(\Omega)$. 
\end{proof}
\subsection{Existence of a radial solution}
\label{sub:existenceradial}
We are now ready to prove one of the central results of the paper, which gives the existence of a compactly supported radial solution for the Problem ($P_\gamma$), introduced in Section~\ref{sub:auxpb}.

\begin{theorem}\label{Thm:AsymptoticProblemStatement} 
Problem ($ P_\gamma$) has at least one solution $U_\gamma\in K_\gamma\cap H^1_\text{rad}(\mathds{R}^N)$ for every $\gamma\in\left]0, +\infty\right[$. Moreover, there exists $\gamma_0=\gamma_0\big(N,W\vert_{[0,s_0]}\big)\in\left]0,+\infty\right[$ such that, for every $\gamma\ge\gamma_0:$
\begin{equation}\label{Eq:AsymptoticProblemStatement0}
E(U_\gamma)<0,
\end{equation}
\begin{equation}\label{Eq:AsymptoticProblemStatement0+}
\int_{\mathds R^N} U_\gamma\, dx=\gamma, 
\end{equation}
and there exists $R_\gamma=R(U_\gamma)>0$ such that 
\begin{equation}\label{Eq:AsymptoticProblemStatement1}
\supp(U_\gamma)=\overline{B_{\mathds{R}^N}(0,R_\gamma)}. 
\end{equation} 
The function $U_\gamma$ is of class $C^{1,\alpha}$ in $\mathds R^N$, for every $\alpha\in\left]0,1\right[$, and it is of class $C^2$ on its positivity set. 
\end{theorem}
\begin{remark}\label{Rem:BigVolumesDependenceFromPotential} Among other things, the above theorem says that $\gamma_0$ depends only on the restriction of $W$ to the interval $[0,s_0]$. This means that a constant $\gamma_0$ as above can be defined also problems ($P_\gamma$) with a potential $W$ that violates the subcritical growth condition \eqref{Eq:Potential0}. This observation will be useful later, see Remark~\ref{Rem:BigVolumesDependenceFromPotential0} below.
\end{remark}
\begin{proof} Without loss of generality we can minimize $E$ over $K_\gamma\cap H^1_\text{rad}(\mathds{R}^N)$. In fact, by Lemma~\ref{Lemma:EnergyDecreasingUnderScwartzSymmetrization}, $E(u^*)\le E(u)$, and $\int_{\mathds R^N} u^* dx=\int_{\mathds R^N}u\, dx$, where $u^*\in H^1_\text{rad}(\mathds{R}^N)$ is the symmetric decreasing rearrangement of $u$ (see Definition \ref{Def:SymmetricDecreasingRearrangement}). 

We will prove that a minimizing sequence $(u_j)$ in $H^1_\text{rad}(\mathds{R}^N)$ is bounded in $H^1(\mathds{R}^N)$. By assumptions \eqref{Eq:Potential}, \eqref{Eq:Potential0} and \eqref{Eq:Potential2} on the potential $W$, we have that $W(s)\ge-ks$ for some $k\in\mathds{R}^+$ and every $s\ge0$. Then, there exists $C\in\mathds{R}$ satisfying:
\begin{equation}
C\ge\int|\nabla u_j|^2+\int W(u_j)\ge\int|\nabla u_j|^2-k\gamma,
\end{equation}
and so:
\begin{equation}
\int|\nabla u_j|^2\le k\gamma+C.
\end{equation}
This says that $\left\vert\nabla u_n\right\vert$ is bounded in $L^2(\mathds{R}^N)$, hence by Sobolev's inequality $u_j$ is bounded in $L^{2^*}(\mathds{R}^N)$. On the other hand, $\int u_j=\gamma$ and $u_j\ge 0$, so $u_j$ is bounded in $L^1(\mathds{R}^N)$ too. Using interpolation, we have that $u_j$ is bounded in $L^2(\mathds{R}^N)$, and so $u_j$ is bounded in $H^1(\mathds{R}^N)$. Thus, up to subsequences, $u_j$ is weakly convergent to a function $U_\gamma\in H^1(\mathds{R}^N)$. The theorem of Strauss \cite{Strauss1977} (see also  \cite[Appendix A.1, Theorem~142]{BenciFortunatoBook}), asserts that for any $N\ge2$, one has a compact inclusion of $H^1_\text{rad}(\mathds{R}^N)$ into $L^p(\mathds{R}^N)$, for every $2<p<2^*=\frac{2N}{N-2}$ ($2^*=+\infty$, when $N=2$). 

Then, by a standard application of Nemytskii's theorem, we have that $u\mapsto W\circ u$ is a compact operator from $H^1_\text{rad}(\mathds{R}^N)$ to its topological dual $\left[H^1_\text{rad}(\mathds{R}^N)\right]'$. In particular, the functional $u\mapsto\int_{\mathds{R}^N} W\big(u(x)\big)\,\mathrm dx$ is weakly continuous. Moreover $u\mapsto\int_{\mathds{R}^N} \left\vert\nabla u\right\vert ^{2}\,\mathrm dx$ is weakly lower-semi-continuous. The direct method of the calculus of variations ensures the existence of a minimum $U_\gamma\in K_\gamma$, since $K_\gamma$ is convex and strongly closed, so \emph{a fortiori} it is also weakly closed. 
 This proves the first assertion of the theorem. \medskip
 
In order to prove \eqref{Eq:AsymptoticProblemStatement0}, let us consider a nonnegative function $\varphi\in H^1_\text{rad}(\mathds{R}^N)$, such that $\int_{\mathds R^N}\varphi=1$, and such that 
\begin{equation}\label{Eq:BigVolumesDependenceFromPotential}
B_\varphi:=\int_{\mathds R^N} W(\varphi)\,\mathrm dx<0.
\end{equation} 
The existence of such a function $\varphi$ follows from assumption \eqref{Eq:Potential2}. Set $A_\varphi=\int_{\mathds{R}^N}\frac{1}{2}|\nabla\varphi|^2\,\mathrm dx>0$. Next, let us set $\varphi_\rho(x)=\varphi\left({x}/{\rho}\right)$. Recalling the definition \eqref{eq:defKgamma}, it is easy to check that
\begin{equation}
\varphi_\rho\in H^1_\text{rad}(\mathds{R}^N)\cap K_{\rho^N},
\end{equation}
and that
\begin{equation}\label{Eq:AsymptoticProblemProof}
E(\varphi_\rho)=A_\varphi\rho^{N-2}+B_\varphi\rho^N.
\end{equation} 
So, for every $\rho\ge\rho_0=\rho_0(n, W):=\sqrt{-\frac{A_\varphi}{B_{\varphi}}}=>0$, we have $E(\varphi_\rho)<0$. 

Set $\gamma_0=\gamma_0(N,W)=\rho_0^N>0$ (observe that actually with the right choice of $\varphi$ in \eqref{Eq:BigVolumesDependenceFromPotential} we have $\gamma_0=\gamma_0(N,W_{|[0,s_0]}$)); then for every $\gamma\ge\gamma_0$, the following inequalities hold
\begin{equation}\label{Eq:AsymptoticProblemProof0-}
E(U_\gamma)\le E(\varphi_{\rho^N})<0.
\end{equation} 
This proves \eqref{Eq:AsymptoticProblemStatement0}.  

In order to prove \eqref{Eq:AsymptoticProblemStatement0+}, we
define $\psi_\rho=U_\gamma({x}/{\rho})$ and we set \[A_{U_\gamma}=\tfrac12\int_{\mathds{R}^N}|\nabla U_\gamma|^2\,\mathrm dx>0,\quad\text{and}\quad B_{U_\gamma}=\int_{\mathds R^N} W(U_\gamma)\,\mathrm dx.\] Then we have 
\begin{equation}\label{Eq:AsymptoticProblemProof-+-}
E(\psi_\rho)=A_{U_\gamma}\rho^{N-2}+B_{U_\gamma}\rho^N.
\end{equation}
From \eqref{Eq:AsymptoticProblemProof-+-} we obtain $\left[E(\psi_\rho)\right]\big\vert_{\rho=1}=A_{U_\gamma}+B_{U_\gamma}=E(U_\gamma)<0$ and so 
\begin{equation}\label{Eq:AsymptoticProblemProof-+-+}
B_{U_\gamma}<0.
\end{equation} 
Using \eqref{Eq:AsymptoticProblemProof0-} and \eqref{Eq:AsymptoticProblemProof-+-+}, we then get:
\begin{eqnarray}\label{Eq:Phozaev}\nonumber
\frac{\mathrm d}{\mathrm d\rho}\left[E\left(\psi_\rho\right)\right]_{|\rho=1} & = & \left[(N-2)\rho^{N-3}A_{U_\gamma}+NB_{U_\gamma}\rho^{N-1}\right]_{\rho=1}\\ 
& = & (N-2)A_{U_\gamma}+NB_{U_\gamma}\\ \nonumber
& = & (N-2)(A_{U_\gamma}+B_{U_\gamma})+2B_{U_\gamma}\\\nonumber 
 & = & (N-2)E(U_\gamma)+2B_{U_\gamma}<0.
\end{eqnarray}
From last inequality it follows that $\int U_\gamma=\gamma$, otherwise for some $\rho>1$ we would have $\int\psi_\rho=\gamma$ and $E(\psi_\rho)<E(U_\gamma)$, this facts being in contradiction with the fact that $U_\gamma$ is an absolute minimum in $K_\gamma$. Thus the proof of \eqref{Eq:AsymptoticProblemStatement0+} is accomplished.
 
It remains to prove \eqref{Eq:AsymptoticProblemStatement1}. First, let us observe that the support of $U_\gamma$ is either $\mathds{R}^N$ or a ball of finite radius centered at the origin. Namely, $U^*_\gamma=U_\gamma$ where $U^*_\gamma$ is the Schwartz's symmetrization, and therefore $U_\gamma$ is radially symmetric.  By our assumptions $U_\gamma$ satisfies equation \eqref{Eq:WeakFormulationOnThePositivitySet} in the weak sense of $W^{1,2}$ in $\Gamma$ (see \cite[p.\ 43]{KinderlehrerStampacchia}), where $\Gamma$ is the positivity set of $U_\gamma$, see \eqref{eq:defGamma}. By standard elliptic regularity results (e.g. \cite[Theorem~9.19]{GilbargTrudinger}), $U_\gamma$ is of class $C^2$ on $\Gamma$.
Moreover, the stationarity condition for $U_\gamma$ yields:
\begin{equation}\label{Eq:AsymptoticEulerLagrangePositivity}
E'(U_\gamma)(U_\gamma)=\int\left[\langle\nabla U_\gamma,\nabla U_\gamma)\rangle+W'(U_\gamma)U_\gamma)\right]\,\mathrm =\lambda_\gamma\int U_\gamma\,\mathrm =\lambda_\gamma\gamma,
\end{equation} 
for some $\lambda_\gamma\in\mathds R$.
But $E'(U_\gamma)(U_\gamma)=\frac{\mathrm d}{\mathrm dt}\left[E(tU_\gamma)\right]_{|t=1}\le0$, since $U_\gamma$ is a minimum in $K_\gamma$. 
Combining \eqref{Eq:Phozaev} with \eqref{Eq:AsymptoticEulerLagrangePositivity} we get readily 
\begin{equation}\label{Eq:AsymptoticProblemProof0-+}
\lambda_\gamma\le0.
\end{equation}
It is easy to show that inequality \eqref{Eq:AsymptoticProblemProof0-+} is strict, for otherwise, using \eqref{Eq:WeakFormulationOnThePositivitySet} we would have $E'(U_\gamma)=0$, which  contradicts \eqref{Eq:Phozaev}.

Using standard \emph{a priori} estimates (see Proposition \ref{thm:bounds}), it is easy to show that $U_\gamma$ is bounded from above,\footnote{For this conclusion, it suffices to assume that $W(s)>W(s_0)$ for $s$ in a right neighbrohood of $s_0$, see Remark~\ref{thm:rembastapocopiugrande}.} namely, $0\le U_\gamma\le s_0$.  Therefore $U_\gamma\in L^{\infty}(\Gamma)$, and the $L^\infty$-norm is bounded uniformly with respect to $\gamma$. 

By standard elliptic regularity (see for instance \cite[Theorem~9.19]{GilbargTrudinger})  $U_\gamma$ is in $C^{2,\alpha}_{\textrm{loc}}(\Gamma)$ for every $\alpha$. 
To deduce that $U_\gamma\in C^{1,\alpha}(\mathds R^N)$ we  apply Theorem \ref{Thm:FreeBoundaryRegularityForDerivative} with $\Omega$ 
equal to $B_{\mathds R^n}(0, R_\gamma+1)$ (for instance) and $f=-W'(U_\gamma)+\lambda_\gamma$, $\psi_1\equiv0$, $\psi_2=\Vert U_\gamma\Vert_{\infty,\mathds R^N}+1\in\left]0,+\infty\right[$. So 
$f\in L^\infty(\mathds R^N)$ and for every $0<\alpha<1$, $\psi_1,\psi_2\in C^{1,\alpha}(\overline{B_{\mathds R^n}(0, R_\gamma+1)})$ 
since $\psi_1,\psi_2$ are constants, then $U_\gamma\in C^{1,\alpha}_{\text{loc}}(\mathds R^N)$, for all $\alpha\in\left[0,1\right[$. We recall from a result contained in 
\cite{Strauss1977} (see also \cite[Lemma~141]{BenciFortunatoBook}) that, since $U_\gamma\in H^1_\text{rad}(\mathds{R}^N)$ the following estimate due to Strauss holds 
\begin{equation}\label{Eq:StraussEstimates}
\big|U_\gamma(x)\big|\le C\,\frac{\Vert U_\gamma\Vert_{H^1(\mathds{R}^N)}}{|x|^{\frac{N-1}{2}}},\quad \text{ for a.e. }x\in\mathds{R}^N,
\end{equation} 
for some positive constant $C$.
Now we argue indirectly and we assume that the support of $U_\gamma$ is $\mathds{R}^N$, i.e., that $\Gamma=\mathds{R}^N$, to obtain a contradiction.
Fix $a\in\left]0,+\infty\right[$ small enough so that $W'(s)\ge ks$ for every $s\in\left]0,a\right]$. This choice is always possible by \eqref{Eq:Potential}. From \eqref{Eq:StraussEstimates} we deduce the existence of $r_0>0$ such that $U_\gamma(x)\le a$ if $|x|\ge r_0$, and then $W\big(U_\gamma(x)\big)\ge0$.
Since $U_\gamma$ is radially symmetric, we can write $U_\gamma(x)=u_\gamma(|x|)$, where $u_\gamma\colon\left[0,+\infty\right[\to\mathds{R}$. Recalling that $U_\gamma\in C^{2,\alpha}(\Omega')$ for every $\Omega'\Subset\Gamma$ equation \eqref{Eq:AsymptoticEulerLagrangeEquality} below 
\begin{equation}\label{Eq:AsymptoticEulerLagrangeEquality}
-\Delta U_\gamma+W'(U_\gamma)=\lambda_\gamma,  
\end{equation}is satisfied in the classical sense, and it gives the following ordinary differential equation for $u_\gamma$:
\begin{equation}\label{Eq:ThmAsymptoticProblemEDO}
\frac{\mathrm d}{\mathrm dr}\left[r^{N-1}u_\gamma'(r)\right]=\left[-\lambda_\gamma+W'\big(u_\gamma(r)\big)\right]r^{N-1},\quad\forall\, r\in\left[0,+\infty\right[.
\end{equation}
Integrating \eqref{Eq:ThmAsymptoticProblemEDO} on the interval $[r_0,r]$ we get 
\begin{eqnarray*}
r^{N-1}u_\gamma'(r)-r_0^{N-1}u_\gamma'(r_0) & = & -\frac{\lambda_\gamma}{N}(r^N-r_0^N)+\int_{r_0}^{r}W'(u_\gamma(s))s^{N-1}ds\\
 & \ge & -\frac{\lambda_\gamma}{N}(r^N-r_0^N)\ge-\frac{\lambda_\gamma}{N}\cdot r^N+c_0,
\end{eqnarray*}
where $c_0\in\mathds{R}$ is a constant independent of $r$. From the last inequality we see that
\begin{equation}
u_\gamma'(r)\ge-\frac{\lambda_\gamma}{N}r+c_1,
\end{equation}
where $c_1\in\mathds{R}$ is independent of $r$.
Integrating again we get 
\begin{equation}
u_\gamma(r)\ge c_2-\frac{\lambda_\gamma}{2N}r^2,
\end{equation} 
with $c_2\in\mathds{R}$ independent of $r$.
Exploiting the fact that $\lambda_\gamma<0$, the above equation contradicts the Strauss's decay estimates \eqref{Eq:StraussEstimates}. This contradictions shows that $\Gamma=B_{\mathds{R}^N}(0,R_\gamma)$, and this concludes the proof. 
\end{proof}
\subsection{Asymptotics for the radius $R_\gamma$}\label{sub:asymptotics}
We need to show that $R_\gamma\cong\gamma^\frac1N$ as $\gamma\to+\infty$. More precisely:
\begin{theorem}\label{Thm:AsymptoticEquivalentOfTheRadius}
In the notations of Theorem~\ref{Thm:AsymptoticProblemStatement}, 
there exist positive constants $\widetilde\gamma_0=\widetilde\gamma_0\left(N, \gamma_0, \Vert W\vert_{[0,s_0]}\Vert_\infty\right)\ge\gamma_0$, $C^-=C^-({W,N})$, and $C^+=C^+({W,N})$  such that the following inequalities hold:
\begin{equation}\label{Eq:AsymptoticProblemStatementStructuralConditionToMakePhotographyWorks}
C^-\,\gamma^{\frac1N}\le R_\gamma<C^+\,\gamma^{\frac1N},
\end{equation} 
for all $\gamma>\widetilde\gamma_0$.
\end{theorem}
\begin{remark} 
The constants $C^+$ and $C^-$ in \eqref{Eq:AsymptoticProblemStatementStructuralConditionToMakePhotographyWorks}
can be estimated as follows:
\begin{equation}\label{eq:explC+}
C^+=\frac32\left(\frac1{s_1\,\omega_N}\right)^{\frac1N},
\end{equation} where  $s_1>0$ is the first positive zero of $W'$:
\begin{equation}\label{eq:defs1}
s_1=\min\big\{s>0:W'(s)=0\big\},
\end{equation}
and 
\begin{equation}\label{eq:explC-}
C^-({W,N})=\left(\frac1{s_0\,\omega_N}\right)^{\frac1N}.
\end{equation}
By our assumptions, $s_1<s_0$, and therefore $C^-<C^+$.
\end{remark}
\begin{proof}
Since $U_\gamma=U^*_\gamma$, from the definition of symmetric decreasing rearrangement that $U_\gamma$ is nonincreasing. From this it follows that $U_\gamma(0)=\Vert U_\gamma\Vert_{\infty}\le s_0$ which implies 
\begin{equation}\label{Eq:RadiusLowerBound}
s_0\omega_NR_\gamma^N\ge U_\gamma(0)\omega_NR_\gamma^N\ge\gamma,
\end{equation}
from which the first inequality in \eqref{Eq:AsymptoticProblemStatementStructuralConditionToMakePhotographyWorks} follows readily for every $\gamma\ge\gamma_0$, with $C^-$ given by \eqref{eq:explC-}.\medskip

Establishing the second inequality in \eqref{Eq:AsymptoticProblemStatementStructuralConditionToMakePhotographyWorks}
requires a much more involved argument, which will take the remainder of this section. 
Towards this goal, let us observe that, since $0$ is a local maximum of $U_\gamma$, $\Delta U_\gamma(0)\le0$, and by \eqref{Eq:AsymptoticEulerLagrangeEquality} $\Delta U_\gamma(0)=-\lambda_\gamma+W'\big(U_\gamma(0)\big)$ with $-\lambda_\gamma>0$, and so $W'\big(U_\gamma(0)\big)<0$, which implies $U_\gamma(0)>s_1>0$, where $s_1$ is given in \eqref{eq:defs1}. 
Set
\[
a=\sup \left\{ t\ |\ W^{\prime }(s)\geq 0, \forall s\in [0,t]\right\}
\]
and
\[
\widetilde{R}_\gamma=\inf \big\{ \vert x\vert: U_\gamma(x)\leq a\big\};
\]
clearly:
\begin{equation}\label{Eq:AsymptoticProblemStatementFirstRadiusEstimate}
s_1\omega_N\widetilde R^N_{\gamma}<\gamma.
\end{equation}
We now want to estimates the real number $z=R_\gamma-\widetilde R_\gamma$. Using elementary Taylor expansion we get:
\begin{eqnarray*}
a =u_\gamma(\widetilde{R}_\gamma)&=&u_\gamma(R_\gamma)+u_\gamma^{\prime }(R_\gamma)\left( \widetilde{R}_\gamma-R_\gamma\right) +
\frac{1}{2}u_\gamma^{\prime \prime }(\theta )\left( \widetilde{R}_\gamma-R_\gamma\right) ^{2}
\label{palla} \\
&=&u_\gamma^{\prime }(R_\gamma)\left( \widetilde{R}_\gamma-R_\gamma\right) +\frac{1}{2}u_\gamma^{\prime \prime
}(\theta )\left( \widetilde{R}_\gamma-R_\gamma\right) ^{2}
\end{eqnarray*}
for some $\theta \in \left]\smash{\widetilde{R}_\gamma,R_\gamma}\right[$.
Our equation \eqref{Eq:ThmAsymptoticProblemEDO} becomes
\[
\frac{\mathrm d}{\mathrm dr}\left[ r^{N-1}u_\gamma^{\prime }(r)\right] =\left[
-\lambda_\gamma +W^{\prime }(u_\gamma(r))\right] r^{N-1},
\]
i.e., 
\[
u_\gamma^{\prime \prime }(r)+\frac{N-1}{r}u_\gamma^{\prime }(r)=-\lambda_\gamma +W^{\prime
}(u_\gamma(r))>-\lambda_\gamma>0,
\]
whenever $r\in \left]\smash{\widetilde{R}_\gamma,R_\gamma}\right[$.
Since $U_\gamma^{\prime }(r)\leq 0$  for $r\in \left]\smash{\widetilde{R}_\gamma},R_\gamma\right[$ (this is a property of symmetric rearrangements), it follows that
\begin{equation}\label{Eq:SecondDerivativeLowerBound}
U_\gamma^{\prime \prime }(r)>-\lambda_\gamma,\quad \forall\, r\in \left]\smash{\widetilde{R}_\gamma,R_\gamma}\right[. 
\end{equation}
We need to give an estimate for a positive lower bound
$\lambda_\gamma \geq w_\gamma^{+}>0$. Towards this goal, we consider the following comparison function $v_\gamma(x)=\tilde v_\gamma(|x|)$, where $\tilde v_\gamma\colon\left[0,+\infty\right[\to\mathds R$ is the piecewise affine function defined by:
\[\tilde v_\gamma(r)=\begin{cases}s_0,&\text{if $r\in[0,t_0]$;}\\
s_0-s_0(r-t_0),&\text{for $r\in[t_0,t_0+1]$};\\0,&\text{if $r>t_0+1$},\end{cases}\]
with the constant $t_0>0$ suitably defined.
It is easy to check that $v_\gamma\in W^{1,2}_0(\mathds{R}^N)$, that we can choose $0<t_0=t_0(\gamma)= c_1\gamma^{\frac1N}$, with $c_1=(\frac{4\omega_Ns_0}3)^{-\frac1N}$, so that 
\begin{equation}\label{Eq:VolumeOfAnsatz}
\tfrac12\gamma\le\int_{\mathds R^N} v_\gamma\,\mathrm dx=\omega_Ns_0t_0^N\left(1+\tfrac12\big[(1+\tfrac1{t_0})^N-1\big]\right)\le\gamma,
\end{equation} 
for large $\gamma$, and $v_\gamma\ge0$. Since $U_\gamma$ is a minimizer for  Problem ($\mathrm P_\gamma$), we have $E[U_\gamma]\le E[v_\gamma]$. On the other hand, an explicit computation of $E[v_\gamma]$ gives: 
\begin{eqnarray}\nonumber
\frac1{\alpha_{N-1}}E[v_\gamma] & =  & \frac12\int_{t_0}^{t_0+1}
\tilde v'_\gamma(r)^2r^{N-1}\,\mathrm dr+\int_0^{t_0} W\big(\tilde 
v_\gamma(r)\big)r^{N-1}\,\mathrm dr\\
 \label{Eq:EnergyOfAnsatz}
&& + \int_{t_0}^{t_0+1}W\big(\tilde v_\gamma(r)\big)r^{N-1}\,\mathrm dr\\  \nonumber
& \le & \left(\tfrac{s_0^2}{2N}+\hat{w}\right)t_0^N\big[(1+\tfrac1{t_0})^N-1\big]-\tfrac mNt_0^N,
\end{eqnarray}
where $\hat{w}:=||W||_{\infty,[0,s_0]}$ (recall that $\alpha_{N-1}$ denotes the area of the unit ball in $\mathds R^N$). We denote by $\bar{E}^*_\gamma$ the last term of the above inequality, and se set $E^*_\gamma:=\alpha_{N-1}\bar{E}^*_\gamma$, so that:
\begin{equation}\label{eq:stimaEUgamma}
 E(U_\gamma)\le E(v_\gamma)\le E_\gamma^*.
 \end{equation}
By \eqref{Eq:VolumeOfAnsatz} , one has $E^*_\gamma<0$ and, for $\gamma$ sufficiently large, $-E^*_\gamma\sim c_2\gamma$ for some positive constant $c_2$.

We now use the classical Pohozaev identity in the bounded starshaped domain $B_{\mathds{R}^N}(0, R_\gamma)$ to obtain:
\begin{equation}\label{Eq:PohozaevDimension}
\begin{aligned}&\int_{B_{\mathds{R}^N}(0, R_\gamma)}-\lambda_\gamma U_\gamma+W(U_\gamma) \\&= \left(\tfrac1N-\tfrac12\right)\int_{B_{\mathds{R}^N}(0, R_\gamma)}\|\nabla U_\gamma\|^2
 -  \tfrac12\int_{\partial B_{\mathds{R}^N}(0, R_\gamma)}(x\cdot\nu)\left(\tfrac{\partial U_\gamma}{\partial\nu}\right)^2,
 \end{aligned}
\end{equation}
where $\nu$ is the outward pointing normal unit field to the boundary of the ball.
On the other hand by \eqref{Eq:AsymptoticProblemProof0-} and \eqref{eq:stimaEUgamma}, we get
\begin{equation}
E[U_\gamma]=-\int [W(U_\gamma)]^-+\int [W(U_\gamma)]^++\frac12\int |\nabla U_\gamma|^2<E^*_\gamma<0.
\end{equation}
Combining the last two equations leads to
\begin{equation}
\begin{aligned}\label{Eq:LagrangeMultiplierLowerBound}
-\lambda_\gamma\gamma & =  \int[W(U_\gamma)]^--\int[W(U_\gamma)]^+\\&+\left(\tfrac1N-\tfrac12\right)\int_{B_{\mathds{R}^N}(0, R_\gamma)}\|\nabla U_\gamma\|^2 -  \tfrac12\alpha_{N-1}R_\gamma^N(u'_\gamma(R_\gamma))^2\\ 
& \ge  \tfrac1N\int_{B_{\mathds{R}^N}(0, R_\gamma)}\|\nabla U_\gamma\|^2-E^*_\gamma-\tfrac12\alpha_{N-1}R_\gamma^N(u'_\gamma(R_\gamma))^2\\
& \ge  -E^*_\gamma-\tfrac12\alpha_{N-1}R_\gamma^N(U'_\gamma(R_\gamma))^2.
\end{aligned}
\end{equation}
Now, we claim that :
\begin{equation}\label{Eq:C1ContactWithObstacle}
U'_\gamma(R_\gamma)=0,
\end{equation} for every $\gamma\ge\gamma_0$.
This follows easily from the $C^{1,\alpha}$-regularity of $U_\gamma$, keeping im mind that $U_\gamma(r)=0$ for $r>R_\gamma$.  Combining \eqref{Eq:LagrangeMultiplierLowerBound} with \eqref{Eq:C1ContactWithObstacle}, and taking $\gamma$ large enough we get 
\begin{equation}
-\lambda_\gamma\ge w_\gamma^*:=-\frac{E^*_\gamma}{\gamma}>0,
\end{equation}
where $w^*>0$ is a positive constant that could be chosen equal to $\frac{m\omega_N}{2}=w^*>0$.
Further, since $U_\gamma^{\prime }(R_\gamma)\leq 0$ and $\left( \widetilde{R}_\gamma-R\right) <0$, we obtain 
\begin{eqnarray*}
a &=&U_\gamma(R_\gamma)+U_\gamma^{\prime }(R)\left( \widetilde{R}_\gamma-R_\gamma\right) +\frac{1}{2}u^{\prime
\prime }(\theta )\left( \widetilde{R}_\gamma-R_\gamma\right) ^{2} \\
&\geq &\frac{1}{2}w^*\left( \widetilde{R}_\gamma-R_\gamma\right) ^{2},
\end{eqnarray*}
from which we deduce 
\[
R_\gamma-\widetilde{R}_\gamma\leq \sqrt{\frac{2a}{w_\gamma^*}},
\]
i.e., 
\begin{equation}\label{Eq:FinalEstimate}
R_\gamma\leq \widetilde{R}_\gamma+\sqrt{\frac{2a}{w^*_\gamma}}\leq k\gamma ^{1/N}+\sqrt{\frac{2a%
}{w^*_\gamma}}.
\end{equation}
The second inequality here follows easily by the estimate of $\widetilde R_\gamma$ given below
\begin{eqnarray*}
\gamma=\int_0^{R_\gamma}U_\gamma(r)\alpha_{N-1}r^{N-1}dr & \ge & \int_0^{\widetilde R_\gamma}U_\gamma(r)\alpha_{N-1}r^{N-1}dr\\
& \ge & a\int_0^{\widetilde R_\gamma}\alpha_{N-1}r^{N-1}dr= a\omega_N\widetilde R_\gamma^N. 
\end{eqnarray*}
Thus 
\begin{equation}
\widetilde R_\gamma\le\left(\frac{\gamma}{a\omega_N}\right)^{\frac1N}=k_{N,W}\gamma^{\frac1N},
\end{equation}
with this last equation we justify \eqref{Eq:FinalEstimate} and indeed we accomplish the proof of the theorem. 
\end{proof}
\begin{remark}\label{Rem:BigVolumesDependenceFromPotential0} 
Using the observation in Remark~\ref{Rem:BigVolumesDependenceFromPotential}, and the inequalities \eqref{Eq:VolumeOfAnsatz}, \eqref{Eq:EnergyOfAnsatz}, it is easy to see that $\widetilde\gamma_0$ depends only on $W\vert_{[0,s_0]}$, which means that $\widetilde\gamma_0$ can be defined also for potentials $W$ that violate the subcritical growth condition \eqref{Eq:Potential0}. This is an important observation in view of a multiplicity result without the assumption of the subcritical growth condition \eqref{Eq:Potential0}.
\end{remark}

\section{Proof of the main results}\label{sec:proofaddhp} 
Consider the open sets:
\begin{equation}
\begin{aligned}\label{eq:defOmegar}
&\Omega^+_r=\left\{x\in\mathds{R}^N:\mathrm{dist}(x,\Omega)<r\right\},&\\
&\Omega^-_r=\big\{x\in\Omega:\mathrm{dist}(x,\partial\Omega)>r\big\},&
\end{aligned}
\end{equation} 
where $\mathrm{dist}$ is the usual Euclidean distance of $\mathds{R}^N$. 
Let $r>0$ be small enough such that both $\Omega^+_r$ are $\Omega^-_r$ are homotopically equivalent to $\Omega$ via some suitable maps \[f_0\colon\Omega\longrightarrow\Omega^-_r\quad\text{and}\quad g_0\colon\Omega^+_r\longrightarrow\Omega.\]
The existence of such $r>0$ and the homotopy equivalences $f_0$, $g_0$, follows from the assumption that $\partial\Omega$ is Lipschitz.
\smallskip

\noindent
A proof of our results is obtained by applying Theorems~\ref{Thm:FotoLS} and \ref{Thm:FotoMorse} to the following setup, 
recalling the constants $\widetilde\gamma_0$ (Theorem~\ref{Thm:AsymptoticProblemStatement}), and $C^+$ (Theorem \ref{Thm:AsymptoticEquivalentOfTheRadius}):
\begin{itemize}
\item $\mathfrak{M}=\mathfrak{M}^{V}$, where
\begin{equation} 
\label{eq:defMV}
\mathfrak{M}^{V}=\left\{ u\in H_{0}^{1}(\Omega ):\int_{\Omega
}u(x)\,\mathrm dx=V\right\}, 
\end{equation}
with 
\begin{equation}\label{Eq:b}
V\le V_1:=\min\left\{\left(\frac r{{C^+}}\right)^N,\left(\frac{r}{R_{\widetilde\gamma_0}}\right)^N\widetilde\gamma_0\right\}=\left(\frac r{{C^+}}\right)^N,
\end{equation}
\item $J={E_{\varepsilon}}\big\vert_{{\mathfrak{M}^V}}$, where 
\begin{equation*}
E_{\varepsilon}(u)=\frac{\varepsilon^{2}}{2}\int_{\Omega} \left\vert
\nabla u\right\vert ^{2}\,\mathrm dx+\int_{\Omega} W\big(u(x)\big)\,\mathrm dx,
\end{equation*}
with 
\begin{equation}\label{Eq:EpsilonUpperBoundInFunctionOfVolume}
\varepsilon\le\varepsilon_1(V):=\left(\frac{V}{\widetilde\gamma_0}\right)^\frac1N;
\end{equation}
\item $X=\Omega$;
\item $f=\Phi_{\varepsilon,V}\circ f_0\colon\Omega\to E_{\varepsilon}^c\cap\mathfrak{M}^{V}$, where $c=\varepsilon^NE\big(U_{V/\varepsilon^N}\big)$ (see \eqref{Eq:AsymptoticProblemStatement0}) and 
\[\Phi_{\varepsilon,V}\colon\Omega^-_r\longrightarrow E_{\varepsilon}^c\cap\mathfrak{M}^{V}\] is the map
$\Omega_r^-\ni x_0\longmapsto\Phi_{\varepsilon,V}^{x_0}\in E_{\varepsilon}^c\cap\mathfrak{M}^{V}$ defined by
\[\phantom{\qquad x\in\Omega.}\Phi_{\varepsilon,V}^{x_0}(x)=U_{V/\varepsilon^N}\left(\frac{x-x_0}{\varepsilon}\right),\qquad x\in\Omega.\]
\item $g=g_0\circ\beta:E_{\varepsilon}^c\cap\mathfrak{M}^{V}\to\Omega$, where $\beta:E_{\varepsilon}^c\cap\mathfrak{M}^{V}\to\Omega^+_r$ is the map defined as follows 
\begin{equation}\label{Eq:DefBarycenter}
\beta(u):=\frac{\int_\Omega x\big|u(x)\big|\,\mathrm dx}{\int_\Omega |u(x)|\,\mathrm dx}.
\end{equation}
Note that:
\begin{equation}\label{eq:barycenterUgamma}
\phantom{,\quad\forall\,\varepsilon,V,x_0.}
\beta\left(\Phi_{\varepsilon,V}^{x_0}\right)=x_0,\quad\forall\,\varepsilon,V,x_0.
\end{equation}
\end{itemize}
Let us now show that all the above objects are well defined, and that, using this framework, the assumptions of Theorems \ref{Thm:FotoLS} and \ref{Thm:FotoMorse} are satisfied. 
\begin{lemma}\label{Lemma:PS} 
For every $\varepsilon>$ and $V>0$, the functional ${E_{\varepsilon}}\big\vert_{\mathfrak{M}^{V}}$ satisfies the Palais-Smale condition.
\end{lemma}
\begin{proof} Assume that $(u_n)$ is a Palais-Smale sequence at level $c$. Observe that, writing equations \eqref{Eq:DefPS} and \eqref{Eq:DefPS0} explicitly, we get:
\begin{equation}\label{Eq:PSLS}
\frac{\varepsilon ^{2}}{2}\int_{\Omega} \left\vert
\nabla u_n\right\vert ^{2}\,\mathrm dx+\int_{\Omega} W\big(u_n(x)\big)\,\mathrm dx\longrightarrow c,\quad\text{as $n\to\infty$,}
\end{equation}
\begin{equation}\label{Eq:PSLS0}
-\varepsilon^2\Delta u_n+W'(u_n)=\lambda_n+T_n,
\end{equation}
where $\lim\limits_{n\to\infty}T_n=0$ strongly in $H^{-1}(\Omega)$. Then by \eqref{Eq:PSLS} and the assumptions \eqref{Eq:Potential}, \eqref{Eq:Potential0}, we obtain
\begin{equation}\begin{aligned} 
c+1 & \ge  \frac{\varepsilon ^{2}}{2}\int_{\Omega} \left\vert\nabla u_n\right\vert ^{2}\,\mathrm dx+\int_{\Omega} W\big(u_n(x)\big)\,\mathrm dx\\
 & \ge   \frac{\varepsilon ^{2}}{2}\int_{\Omega} \left\vert\nabla u_n\right\vert ^{2}\,\mathrm dx-k\int_\Omega u\,\mathrm dx\\
 & =  \frac{\varepsilon ^{2}}{2}\int_{\Omega} \left\vert\nabla u_n\right\vert ^{2}\mathrm dx-kV.
\end{aligned}
\end{equation}
Then $\left\vert\nabla u_n\right\vert$ is bounded in $L^2$ and hence by the Poincar\'e inequality, $u_n$ is bounded in $H_0^1(\Omega)$, so there exists $u\in H_0^1(\Omega)$ such that $u_n\rightharpoonup u$. We have to show that $u_n\to u$ strongly in $\mathfrak{M}_{\varepsilon,c}^{V}$.  It is well known (Nemytskii's theorem) that by \eqref{Eq:Potential0}, the map
\[W'\colon H_0^1(\Omega)\ni u\longmapsto W\circ u\in H^{-1}(\Omega)\]
is a compact nonlinear operator. 
Thus $W'(u_n)\to W'(u)$ strongly in $H^{-1}(\Omega)$. Multiplying \eqref{Eq:PSLS0} by $u_n$ and using the constraints $\int u=V$ we get that $\lambda_n$ is a bounded sequence. So, up to a subsequence, we can assume that $\lambda_n\to\lambda$. 

Now, recalling that $\Delta^{-1}\colon H^{-1}(\Omega)\to H_0^1(\Omega)$ is an isomorphism, we obtain that 
\[u_n=\tfrac1{\varepsilon^2}(-\Delta^{-1})\left[\lambda_n-W'(u_n)+T_n\right]\] is a convergent sequence in $\mathfrak M^V$. This concludes the proof.
\end{proof}

For any open set $\mathcal{U}\subseteq\mathds{R}^N$, we denote by $\overline{\mathcal U}$ its closure, and we define
\[\mathfrak{M}^{V}(\mathcal{U})=\Big\{u\in H^1(\mathds{R}^N):\supp(u)\subseteq\overline{\mathcal{U}},\ \smallint u=V\Big\}.\]
Moreover, we set:
\begin{equation}\label{eq:defenergysublevel}
\begin{aligned}
& m(\varepsilon,\rho, V)=\inf\big\{E_{\varepsilon}(u):u\in\mathfrak{M}^{V}\big(B_\rho(0)\big),\ u\ge0\big\},\\[.2cm]
& 
m^*(\varepsilon,\rho, V)=\inf\big\{E_{\varepsilon}(u):u\in\mathfrak{M}^{V}\big(\mathds{R}^N\setminus B_\rho(0)\big),\ \beta(u)=0,\ u\ge0\big\},
\end{aligned}
\end{equation}
where $\beta(u)$ is defined in \eqref{Eq:DefBarycenter}.
\begin{lemma} For every $V\in\left]0,V_1\right[$, $\varepsilon\in\left]0,\varepsilon_1(V)\right[$ and for all $\rho>0$, the following  inequality holds: 
\begin{equation}\label{Eq:ForLemmaBarycentcerStatement}
m^*(\varepsilon, \rho, V)>m(\varepsilon,\varepsilon R_\gamma, V),\end{equation}
where $\gamma=\frac V{\varepsilon^N}$ and $R_\gamma$ is the radius of the closed ball that supports $U_\gamma$, see \eqref{Eq:AsymptoticProblemStatement1}.
\end{lemma}
\begin{proof} Let $\widetilde{\gamma}_0>0$ be as in Theorem~\ref{Thm:AsymptoticProblemStatement}; by \eqref{Eq:AsymptoticProblemStatement1}, for all $\rho>0$ and for all $\gamma\ge\gamma_0$, the following holds:
\begin{eqnarray*}
 E(U_\gamma) & = & \min\left\{E(u):u\in\mathfrak{M}^{\gamma}\big(B_{R_\gamma}(0)\big),\ u\ge0\right\}\\
&=& \min\left\{E(u):u\in\mathfrak{M}^{\gamma}(\mathds R^N),\ u\ge0\right\}\\
& \le & \min\left\{E(u):u\in\mathfrak{M}^{\gamma}\big(\mathds{R}^N\setminus B_{\rho}(0)\big),\ \beta(u)=0,\ u\ge0\right\}.
\end{eqnarray*}
Next, we will show that this inequality is strict. We argue indirectly and we assume that $w$ is a minimizer of $E$ over the set 
\[\left\{u\in\mathfrak{M}^{\gamma}\big(\mathds{R}^N\setminus B_{\rho}(0)\big): \beta(u)=0,\ u\ge0\right\},\] and that
\[E(U_{\gamma})=E(w).\] 
Thus, $w$ is a map with barycenter at $0$, and with support contained in the exterior of a ball centered at $0$.
Denote by $w^*$ the symmetric decreasing rearrangement of $w$ in $\mathds{R}^N$, see Definition~\ref{Def:SymmetricDecreasingRearrangement}. Clearly, $\beta(w^*)=0$, because $w^*$ is radially symmetric, and $w^*\ne w$, because the support of a decreasing rearrangement is always a ball centered at the orgin. By Lemma~\ref{Lemma:EnergyDecreasingUnderScwartzSymmetrization}, $E(w^*)\le E(w)$. We cannot have $E(w^*)=E(w)$, because if such equality holds, then by Theorem~\ref{Thm:BrothersZiemer} (whose application is allowed by the fact that a classical result of Gidas-Ni-Nirenberg, i.e., Theorem \ref{Thm:GidasNiNirenberg} ensures the validity of \eqref{Eq:BrothersZiemerStatement0}) we would have $w^*=w(\cdot+x_0)$, and so $\beta(w^*)=\beta\big(w(\cdot+x_0)\big)=x_0=0$, which contradicts the fact that $w\ne w^*$. 
This implies $E(w^*)<E(w)$, which gives the following contradiction:
\[E(U_{\gamma})=E(w)>E(w^*)\ge E(U_{\gamma}),\] 
and therefore it shows that:
\begin{equation}\label{eq:epsilon1+}
m^*(1,\rho, \gamma)>m(1, R_\gamma, \gamma).
\end{equation}

Given $u\in H^1(\mathds R^N)$, set $u_\varepsilon(x):=u\left(\varepsilon x\right)$. It is immediate to see that $E_\varepsilon(u)=\varepsilon^NE(u_\varepsilon)$, which implies immediately:
\begin{equation}\label{eq:Eeps}
m(\varepsilon,r, V)=\varepsilon^N\,m\left(1,\tfrac r\varepsilon, \tfrac V{\varepsilon^N}\right)\quad\text{and}\quad
m^*(\varepsilon,r, V)=\varepsilon^N\,m^*\left(1,\tfrac r\varepsilon, \tfrac V{\varepsilon^N}\right).
\end{equation}
Set $\gamma=\frac V{\varepsilon^N}$; by our choice of $\varepsilon_1$, it is $\gamma\ge\tilde{\gamma}_0$,
thus inequality \eqref{Eq:ForLemmaBarycentcerStatement} follows readily from \eqref{eq:epsilon1+} and \eqref{eq:Eeps}. Namely:
\begin{eqnarray*}
m^*(\varepsilon,r,V) & \stackrel{\text{by \eqref{eq:Eeps}}}{=} & \varepsilon^Nm^*\left(1,\tfrac r\varepsilon, \tfrac V{\varepsilon^N}\right)\\
& \stackrel{\text{by \eqref{eq:epsilon1+}}}{>} & \varepsilon^Nm\left(1,R_\gamma, \tfrac V{\varepsilon^N}\right)\\
& \stackrel{\text{by \eqref{eq:Eeps}}}{=} & m(\varepsilon, \varepsilon R_\gamma, V).\qedhere
\end{eqnarray*}
\end{proof}
\begin{remark} From Theorem \ref{Thm:AsymptoticProblemStatement} it is easy to check that $m(1,\rho, \gamma)=E(U_\gamma)$ for every $\rho\ge R_\gamma$. 
\end{remark}
\begin{lemma}\label{Lemma:WellPosedf}  
Given $\varepsilon\in\left]0,\varepsilon_1(V)\right[$, $V\in\left]0,V_1\right[$, and setting:
\begin{equation}\label{eq:decc}
c=c(\varepsilon,V, N,W)=m(\varepsilon, \varepsilon R_\gamma, V),
\end{equation}
where $\gamma=\frac V{\varepsilon^N}$ and $R_\gamma$ is as in \eqref{Eq:AsymptoticProblemStatement1}, then $E_{\varepsilon}^c\cap\mathfrak{M}^{V}$ in nonempty, and the map $f\colon\Omega\to E_{\varepsilon}^c\cap\mathfrak{M}^{V}$ is well defined.
\end{lemma}
\begin{proof} 
By \eqref{Eq:EpsilonUpperBoundInFunctionOfVolume}, $\gamma>\widetilde\gamma_0$, and we obtain:
\[\varepsilon\cdot R_\gamma\stackrel{\text{by \eqref{Eq:AsymptoticProblemStatementStructuralConditionToMakePhotographyWorks}}}<\varepsilon\cdot C^+\cdot\gamma^\frac1N=C^+\cdot V^\frac1N\le C^+\cdot V_1^\frac1N\stackrel{\text{by \eqref{Eq:b}}}<r.\] 
From this inequality and the definition of $\Phi_{\varepsilon,\gamma}$, it is immediate to deduce that \[\supp\big(\Phi_{\varepsilon,\gamma}^{x_0}\big)=\overline{B_{\mathds{R}^N}(x_0,\varepsilon R_\gamma)}\subseteq\Omega\] for every $x_0\in\Omega^-_r$. Now, using an elementary change of variables in the integrals we obtain: 
\begin{multline*}
E_\varepsilon\big(\Phi_{\varepsilon,\gamma}^{x_0}\big) = \frac{\varepsilon^{2}}{2}\int_{\Omega} \left\vert
\nabla_x U_{\gamma}\left(\tfrac{x-x_0}{\varepsilon}\right)\right\vert ^{2}\,\mathrm dx+\int_{\Omega} W\left(U_{\gamma}\left(\tfrac{\ x-x_0}{\varepsilon}\right)\right)\,\mathrm dx\\
= \frac{\varepsilon^{2}}{2}\int_{B_{\mathds{R}^N}(0,R_\gamma)} \big\vert
\nabla_y U_{\gamma}\left(y\right)\big\vert ^{2}\varepsilon^{N-2}\;\mathrm dy+ \int_{B_{\mathds{R}^N}(0,R_\gamma)} W\big(U_{\gamma}\left(y\right)\big)\varepsilon^N\,\mathrm dy\\
= \varepsilon^NE[U_\gamma]=\varepsilon^N\,m(1,R_\gamma, \gamma)= m(\varepsilon, \varepsilon R_\gamma, V)=c,
\end{multline*}
and 
\begin{equation}
\int_{\mathds R^N}\Phi_{\varepsilon,\gamma}^{x_0}\,\mathrm dx=\varepsilon^N\int_{\mathds R^N} U_{\gamma}\,\mathrm dx=\varepsilon^N\gamma=V.
\end{equation}
Hence, $\Phi_{\varepsilon,\gamma}^{x_0}\in E_{\varepsilon}^c\cap\mathfrak{M}^{V}$, (in particular  $E_{\varepsilon}^c\cap\mathfrak{M}^{V}\ne\emptyset$) and we are done.
\end{proof}
\begin{lemma}\label{Lemma:WellPosedBarycenters} For $V\in\left]0,V_1\right[$, $\varepsilon\in\left]0,\varepsilon_1(V)\right[$, the function $g\colon E_{\varepsilon}^c\cap\mathfrak{M}^{V}\to\Omega$ is well defined, i.e., if $u\in\mathfrak{M}^V$, $E_{\varepsilon}(u)\le c=m(\varepsilon, \varepsilon R_\gamma, V)
$ where $\gamma=\frac V{\varepsilon^N}$, we have $\beta(u)\in\Omega^+_r$. 
\end{lemma}
\begin{proof} 
Let us argue by contradiction, assuming that there exists $\overline {u}\in E_{\varepsilon}^c\cap\mathfrak{M}^{V}$ such that $\overline{x}:=\beta(\overline{u})\notin\Omega^+_r$. Then, $\Omega\subset\mathds{R}^N\setminus B_r(\bar{x})$, and therefore
\[m^*(\varepsilon,r, V)\le E_{\varepsilon}(\bar{u})\le c=m(\varepsilon, \varepsilon R_\gamma, V).
\] 
This contradicts \eqref{Eq:ForLemmaBarycentcerStatement}, and concludes the proof. 
\end{proof}
We are now ready to finalize the proof of our main results.
\begin{proof}[Proof of Theorem~\ref{Thm:Main1}] 
It is sufficient to verify assumptions \eqref{Assumption:Thm:FotoLS}, \eqref{Assumption:Thm:FotoLS0}, \eqref{Assumption:Thm:FotoLS1} of Theorems \ref{Thm:FotoLS} and \ref{Thm:FotoMorse} in our variational framework.
For assumption~\eqref{Assumption:Thm:FotoLS} see \eqref{eq:Eeboundedbelow}. Assumption~\eqref{Assumption:Thm:FotoLS0} follows from Lemma \ref{Lemma:PS}. Assumptions~\eqref{Assumption:Thm:FotoLS1} follows from Lemmas \ref{Lemma:WellPosedf} and \ref{Lemma:WellPosedBarycenters}.
As to the last statement of Theorem~\ref{Thm:FotoLS}, note that $\mathfrak M^V$ is contractible. Namely, it is an affine (closed) subspace of $H^1_0(\Omega)$, see \eqref{eq:defMV}.
\end{proof}
\begin{proof}[Proof of Proposition~\ref{thm:boundsenergyindex}]
For every $V\in\left]0,V_1\right[$ and all $\varepsilon\in\left]0,\varepsilon(V_1)\right[$, $\cat(\Omega)$ (resp., $P_1(\Omega)$ in the nondegenerate case) solutions of  problem ($\mathrm P_{V,\varepsilon}$) are found in the energy sublevel $m(\varepsilon,\varepsilon R_{\gamma(\varepsilon,V)},V)$, where $\gamma(\varepsilon,V)=\frac V{\varepsilon^N}$, recall formula~\eqref{eq:defenergysublevel}. Thus, a proof of Proposition~\ref{thm:boundsenergyindex} is obtained by showing that \[\limsup\limits_{\varepsilon\to0}m(\varepsilon,\varepsilon R_{\gamma(\varepsilon,V)},V)<+\infty.\]
This follows readily from the very definition of $m(\varepsilon,\rho,V)$, see \eqref{eq:defenergysublevel}, observing that, by \eqref{Eq:AsymptoticProblemStatementStructuralConditionToMakePhotographyWorks}, the quantity $\varepsilon R_{\gamma(\varepsilon,V)}$ is bounded as $\varepsilon\to0$. In the nondegenerate case, the statement about the boundedness of the Morse index of the low energy solutions follows readily from the observation in Remark~\ref{rem:morseindexlowenergy}.
\end{proof}
\appendix
\section{Auxiliary results: a priori estimates}\label{sec:newassumpt}
\label{sec:apriori}
For the reader's convenience, in this appendix we give the statement and a short proof of some a priori estimates for solutions of elliptic PDE's, that were used in the paper.

Let us consider the elliptic PDE:
\begin{equation}\label{eq:PDE}
-\Delta u+G'(u)=0
\end{equation}
on a bounded domain $\Omega\subset\mathds R^n$, with $G\colon\mathds R\to\mathds R$ a function of class $C^2$ satisfying $G(0)=0$ and
\begin{equation}\label{eq:hpG}
\big\vert G'(s)\big\vert\le A+B\vert s\vert^{p-1},
\end{equation} 
for some positive constants $A,B$ and for some $p<\dfrac{2n}{n+2}$.\smallskip

A weak solution $u\in H^1_0(\Omega)$ of \eqref{eq:PDE} is a critical point\footnote{By standard elliptic regularity, such a weak solution $u$ belongs to $H^3(\Omega)$.} of the functional $E\colon H^1_0(\Omega)\to\mathds R$ defined by:
\begin{equation}\label{eq:energy}
E(u)=\int_\Omega\left[\tfrac12\Vert\nabla u\Vert^2+G(u)\right]\,\mathrm dx,
\end{equation}
i.e., it satisfies:
\begin{equation}\label{eq:weaksol}
\phantom{ v\in H_0^1(\Omega).}\mathrm dE(u)v=\int_\Omega\left[\nabla u\cdot\nabla v+G'(u)v\right]\,\mathrm dx=0,\qquad\forall\, v\in H_0^1(\Omega).
\end{equation}
Assumption \eqref{eq:hpG} implies that $E$ is a well defined $C^1$-functional in $H^1_0(\Omega)$ and that $\mathrm dE(u)$ in \eqref{eq:weaksol} is a bounded linear operator on $H^1_0(\Omega)$.
\begin{proposition}\label{thm:bounds}
Let $u\in H^1_0(\Omega)$ be a solution of \eqref{eq:weaksol}.
Assume that there exists $s_-<0$
\begin{equation}\label{eq:hpderG}
G'(s)<0\quad\text{for}\ s\le s_-
\end{equation}
Then, \[\phantom{,\quad\text{a.e.\ in}\ \Omega.}s_-\le u\,\quad\text{a.e.\ in}\ \Omega.\]
Similarly, if there exists $s_+>0$ such that \begin{equation}\label{eq:hpderG2} G'(s)>0,\quad \text{for $s\ge s_+$},\end{equation}
then \[
\phantom{,\quad\text{a.e.\ in}\ \Omega.}u\le s_+,\quad\text{a.e.\ in}\ \Omega.\]
\end{proposition}
\begin{proof}
For $v\colon\Omega\to\mathds R$, denote by $v^+$ and $v^-$ the nonnegative functions defined by $v^-(x)=\max\big\{-v(x),0\big\}$ and $v^+(x)=\max\big\{v(x),0\big\}$. Define:
\[\Omega_-=\big\{x\in\Omega:u\le s_-\big\}, \quad \Omega_+=\big\{x\in\Omega:u\ge s_+\big\}.\]
Set $v=(u-s_-)^-$; since $u\in H_0^1(\Omega)$, then also $v\in H^1_0(\Omega)$. Plugging such $v$ in \eqref{eq:weaksol}, we get:
\[0=\int_\Omega\nabla u\cdot\nabla(u-s_-)^-+G'(u)(u-s_-)^-\,\mathrm dx=
\int_{\Omega_-}\Vert\nabla u\Vert^2+G'(u)(u-s_-)\,\mathrm dx.\]
If $u\le s_-$ on a set of positive measure, then by \eqref{eq:hpderG} the last integral is strictly positive, giving a contradiction. Thus, $u\ge s_-$ almost everywhere.
\smallskip

Similarly, now plug $v=(u-s_+)^+$ into \eqref{eq:weaksol}:
\[0=\int_\Omega\nabla u\cdot\nabla(u-s_+)^-+G'(u)(u-s_+)^-\,\mathrm dx=
\int_{\Omega_+}\Vert\nabla u\Vert^2+G'(u)(u-s_+)\,\mathrm dx.\]
If $u\ge s_+$ on a set of positive measure, then by \eqref{eq:hpderG2} the last integral is strictly positive, giving a contradiction. Thus, $u\le s_+$ almost everywhere, which concludes the proof.
\end{proof}
\begin{remark}\label{thm:rembastapocopiugrande}
The assumptions of Proposition~\ref{thm:bounds} can be somewhat weakened if one wants to obtain bounds only for solutions $u$ of \eqref{eq:weaksol} that are \emph{minima} of the corresponding energy functional $E$ in \eqref{eq:energy}. Namely, in order to conclude that $s_-\le u$, it is not necessary to assume \eqref{eq:hpderG}. It suffices to assume that $G(s)>G(s_-)$ for $s$ in a left neighborhood of $s_-$, for in this case, if $u<s_-$ somewhere, then the function $u^-\in H^1_0(\Omega)$ defined by $u^-(x)=\max\big\{u(x),s_-\big\}$ would satisfy $E(u^-)<E(u)$, contradicting the minimality assumption for $u$. Similarly, in order to conclude that $u\le s_+$ it suffices to assume that $G(s)>G(s_+)$ for $s$ in a right neighborhood of $s_+$.
\end{remark}
Let us now consider the eigenvalue equation on $\Omega$:
\begin{equation}\label{eq:PDElambda}
-\varepsilon^2\Delta u+W'(u)=\lambda,
\end{equation}
for some $\lambda\in\mathds R$.
\begin{proposition}\label{thm:lowerestlambda}
Let $\lambda\in\mathds R$ and $u\in H_0^1(\Omega)$ be a (weak) solution of \eqref{eq:PDElambda}, with $\int_\Omega u\,\mathrm dx>0$, and satisfying $u\le s_+$ for some $s_+>0$. Then:
\[\lambda\ge w^-:=\min_{s\in[0,s_+]}W'(s).\]
\end{proposition}
\begin{proof}
The function $u$ satisfies:
\begin{equation}\label{eq:weaksoleqn}
\phantom{\quad\forall\,v\in H^1_0(\Omega).}
\varepsilon^2\int_\Omega\nabla u\cdot\nabla v\,\mathrm dx+\int_\Omega W'(u)\,v\,\mathrm dx=\lambda\,v,\quad\forall\,v\in H^1_0(\Omega).
\end{equation}
Denote by $\Omega^+=\big\{x\in\Omega:u(x)\ge0\big\}$ and observe that $\vert\Omega^+\vert>0$, because $\int_\Omega u\,\mathrm dx>0$. 
Setting $v=u^+$ in \eqref{eq:weaksoleqn} we get:
\begin{multline*}\lambda\int_{\Omega^+}u\,\mathrm dx=\varepsilon^2\int_{\Omega^+}\big\Vert\nabla u\big\Vert^2\,\mathrm dx+\int_{\Omega^+}W'(u)u\,\mathrm dx
\ge\int_{\Omega^+}W'(u)u\,\mathrm dx\\\ge w^-\int_{\Omega^+}u\,\mathrm dx.
\end{multline*}
The conclusion follows readily.
\end{proof}

\bibliographystyle{RGB-bibtex}
\bibliography{BiblioBNP}

\begin{thebibliography}{10}

\bibitem{AllenCahn1979}
{\sc S.~M. Allen and J.~W. Cahn}, {\em A macroscopic theory for antiphase
  boundary motion and its application to antiphase domain coarsening}, Acta.
  Metal., 27 (1979), 1085--1095.

\bibitem{BenciNewApproach}
{\sc V.~Benci}, {\em Introduction to {M}orse theory: a new approach}, in
  Topological nonlinear analysis, vol.~15 of Progr. Nonlinear Differential
  Equations Appl., Birkh\"auser Boston, Boston, MA, 1995, 37--177.

\bibitem{BenciCeramiCalcVar}
{\sc V.~Benci and G.~Cerami}, {\em Multiple positive solutions of some elliptic
  problems via the {M}orse theory and the domain topology}, Calc. Var. Partial
  Differential Equations, 2 (1994), 29--48.

\bibitem{BenciCeramiPassaseo}
{\sc V.~Benci, G.~Cerami, and D.~Passaseo}, {\em On the number of the positive
  solutions of some nonlinear elliptic problems}, in Nonlinear analysis, Sc.
  Norm. Super. di Pisa Quaderni, Scuola Norm. Sup., Pisa, 1991, 93--107.

\bibitem{BenciFortunatoBook}
{\sc V.~Benci and D.~Fortunato}, {\em Variational methods in nonlinear field
  equations}, Springer Monographs in Mathematics, Springer, Cham, 2014.
\newblock Solitary waves, hylomorphic solitons and vortices.

\bibitem{BenNarOsoPic2018}
{\sc V.~Benci, S.~Nardulli, L.~E. Osorio~Acevedo, and P.~Piccione}, {\em
  Existence of multiple constant mean curvature boundaries of small enclosed
  volume in compact riemannian manifolds},  (2018).
\newblock In prepation.

\bibitem{BrothersZiemer}
{\sc J.~E. Brothers and W.~P. Ziemer}, {\em Minimal rearrangements of {S}obolev
  functions}, J. Reine Angew. Math., 384 (1988), 153--179.

\bibitem{CahnHilliard1958}
{\sc J.~W. Cahn and J.~E. Hilliard}, {\em Free energy of a non-uniform system
  i:interfacial energy}, J. Chem. Phys., 27 (1958), 258--266.

\bibitem{Campanato65}
{\sc S.~Campanato}, {\em Equazioni ellittiche del {${\rm II}$} ordine espazi
  {${L}\sp{(2,\lambda )}$}}, Ann. Mat. Pura Appl. (4), 69 (1965), 321--381.

\bibitem{Eisen83}
{\sc G.~Eisen}, {\em On the obstacle problem with a volume constraint},
  Manuscripta Math., 43 (1983), 73--83.

\bibitem{Giaquinta2016multiple}
{\sc M.~Giaquinta}, {\em Multiple Integrals in the Calculus of Variations and
  Nonlinear Elliptic Systems.(AM-105)}, vol.~105, Princeton University Press,
  2016.

\bibitem{GidasNiNirenberg}
{\sc B.~Gidas, W.~M. Ni, and L.~Nirenberg}, {\em Symmetry and related
  properties via the maximum principle}, Comm. Math. Phys., 68 (1979),
  209--243.

\bibitem{GilbargTrudinger}
{\sc D.~Gilbarg and N.~S. Trudinger}, {\em Elliptic partial differential
  equations of second order}, Classics in Mathematics, Springer-Verlag, Berlin,
  2001.
\newblock Reprint of the 1998 edition.

\bibitem{HutchinsonTonegawa2000}
{\sc J.~E. Hutchinson and Y.~Tonegawa}, {\em Convergence of phase interfaces in
  the van der {W}aals-{C}ahn-{H}illiard theory}, Calc. Var. Partial
  Differential Equations, 10 (2000), 49--84.

\bibitem{KinderlehrerStampacchia}
{\sc D.~Kinderlehrer and G.~Stampacchia}, {\em An introduction to variational
  inequalities and their applications}, vol.~31 of Classics in Applied
  Mathematics, Society for Industrial and Applied Mathematics (SIAM),
  Philadelphia, PA, 2000.
\newblock Reprint of the 1980 original.

\bibitem{Maggi2012}
{\sc F.~Maggi}, {\em Sets of finite perimeter and geometric variational
  problems}, vol.~135 of Cambridge Studies in Advanced Mathematics, Cambridge
  University Press, Cambridge, 2012.
\newblock An introduction to geometric measure theory.

\bibitem{MasseyBook}
{\sc W.~S. Massey}, {\em Homology and cohomology theory}, Marcel Dekker, Inc.,
  New York-Basel, 1978.
\newblock An approach based on Alexander-Spanier cochains, Monographs and
  Textbooks in Pure and Applied Mathematics, Vol. 46.

\bibitem{Modica1987}
{\sc L.~Modica}, {\em The gradient theory of phase transitions and the minimal
  interface criterion}, Arch. Rational Mech. Anal., 98 (1987), 123--142.

\bibitem{Murray1981}
{\sc J.~D. Murray}, {\em Pre-pattern formation mechanism for animal coat
  markings}, J. Theoret. Biol., 88 (1981), 161--199.

\bibitem{PacardRitore2003}
{\sc F.~Pacard and M.~Ritor\'e}, {\em From constant mean curvature
  hypersurfaces to the gradient theory of phase transitions}, J. Differential
  Geom., 64 (2003), 359--423.

\bibitem{PetrosyanShahgholianUraltseva}
{\sc A.~Petrosyan, H.~Shahgholian, and N.~Uraltseva}, {\em Regularity of free
  boundaries in obstacle-type problems}, vol.~136 of Graduate Studies in
  Mathematics, American Mathematical Society, Providence, RI, 2012.

\bibitem{Presutti2009}
{\sc E.~Presutti}, {\em Scaling limits in statistical mechanics and
  microstructures in continuum mechanics}, Theoretical and Mathematical
  Physics, Springer, Berlin, 2009.

\bibitem{RabinowitzCBMS1986}
{\sc P.~H. Rabinowitz}, {\em Minimax methods in critical point theory with
  applications to differential equations}, vol.~65 of CBMS Regional Conference
  Series in Mathematics, Published for the Conference Board of the Mathematical
  Sciences, Washington, DC; by the American Mathematical Society, Providence,
  RI, 1986.

\bibitem{Strauss1977}
{\sc W.~A. Strauss}, {\em Existence of solitary waves in higher dimensions},
  Comm. Math. Phys., 55 (1977), 149--162.

\bibitem{VanSchaftingen}
{\sc J.~Van~Schaftingen}, {\em Symmetrization and minimax principles}, Commun.
  Contemp. Math., 7 (2005), 463--481.

\end{thebibliography}

\end{document}